\documentclass[a4paper]{article}
\usepackage{graphicx} 
\usepackage{amssymb}
\usepackage{verbatim}
\usepackage{lipsum}
\usepackage{amsfonts}
\usepackage{graphicx}
\usepackage{epstopdf}
\usepackage[ruled,vlined]{algorithm2e}
\usepackage{amsmath,amsthm}
\usepackage{verbatim}
\usepackage{graphics}
\usepackage{graphicx}
\usepackage{subfigure}
\usepackage{bm}
\usepackage{multirow,multicol}
\usepackage{float}
\usepackage{appendix}
\usepackage{chngcntr}
\usepackage{color}
\usepackage{tikz}
\usepackage{pgfplots}
\usepackage{booktabs,hyperref}
\usepackage{longtable}
\usetikzlibrary{arrows.meta}
\usetikzlibrary{shapes.geometric,patterns,hobby}
\ifpdf
\DeclareGraphicsExtensions{.eps,.pdf,.png,.jpg}
\else
\DeclareGraphicsExtensions{.eps}
\fi
\newtheorem{remark}{Remark}
\newtheorem{Theorem}{Theorem}[section]

\newtheorem{Lemma}[Theorem]{Lemma}

\newtheorem{Proposition}[Theorem]{Proposition}

\newtheorem{Definition}[Theorem]{Definition}

\newcommand{\dx}{\,\mathrm{d} x}
\newcommand{\ds}{\,\mathrm{d}s}

\DeclareMathAlphabet{\mathsfsl}{OT1}{cmss}{m}{sl}

\renewcommand{\vec}[1]{\mbox{\boldmath$#1$}}

\newcommand{\md}{\mathrm{D}}

\newcommand{\vu}{\vec{u}}

\pgfplotsset{compat=1.17} 
\usepackage{geometry}
\begin{document}
\title{
Shape Optimization of Supercapacitor Electrode to Maximize Charge Storage\thanks{This work was supported in part by the National Key Basic Research Program under grant 2022YFA1004402, National Key R\&D Program of China (pp. 2023YFF1204201), the National Natural Science Foundation of China under grant (pp. 12401534, pp. 12471377, and pp. 12171319), the China Postdoctoral Science Foundation (pp. 2024M751947), the Postdoctoral Fellowship Program (Grade B) of China Postdoctoral Science Foundation (pp. GZB20240436), and the Science and Technology Commission of
Shanghai Municipality (pp. 22ZR1421900 and 22DZ2229014).}}
\author{Jiajie Li \thanks{School of Mathematical Sciences, MOE-LSC, CMA-Shanghai, and Shanghai Center for Applied Mathematics,  Shanghai Jiao Tong University, Shanghai, China. E-mail: lijiajie199477@sjtu.edu.cn}
\and Shenggao Zhou \thanks{School of Mathematical Sciences, MOE-LSC, CMA-Shanghai, and Shanghai Center for Applied Mathematics,  Shanghai Jiao Tong University, Shanghai, China. Email: sgzhou@sjtu.edu.cn}
\and Shengfeng Zhu \thanks{Key Laboratory of MEA (Ministry of Education) \& Shanghai Key Laboratory of Pure Mathematics and Mathematical Practice \& School of Mathematical Sciences, East China Normal University, Shanghai 200241, China. E-mail: sfzhu@math.ecnu.edu.cn}}
\maketitle

\begin{abstract}
This work proposes a shape optimization approach for electrode morphology to maximize charge storage in supercapacitors. The ionic distributions and electric potential are {modeled} by the steady-state Poisson--Nernst--Planck system. Shape sensitivity analysis is performed to derive the Eulerian derivative in both {volumetric} and boundary expressions. An optimal electrode morphology is {obtained through} gradient flow algorithms. The steady-state Poisson--Nernst--Planck system is efficiently solved by the Gummel fixed-point scheme with finite-element discretization, in which exponential coefficients with large variation are tackled with inverse averaging techniques. {Extensive} numerical experiments are performed to demonstrate the effectiveness of the proposed optimization model and corresponding numerical methods in enhancing charge storage in supercapacitors. It is expected that the proposed shape optimization approach provides a promising tool in the design of electrode morphology from a perspective of charge storage enhancement. 

\end{abstract}



\section{Introduction}
‌The intense demand for clean energy has greatly driven the development of electrochemical technologies. Supercapacitors~\cite{supercapacitor_science_2}, also known as electric double layer capacitors, are promising electrical energy storage devices with wide application spectrum including regenerative braking, energy harvesting, purification, and electric vehicles~\cite{EDLC_app,Supercapacitors_ACS_EL_3}. In contrast to traditional dielectric capacitors, supercapacitors, which store electric energy in electric double layers (EDLs) forming at electrode-electrolyte interface, can {achieve} significantly larger energy densities due to small charge separation in nanoscale EDLs~\cite{Conway::1999}. Such storage mechanism that is absent of chemical reactions brings several unique features, such as higher charging/discharging efficiency, higher power density, and longer cycle life ~\cite{Ramon_JCIS18,Burke00, JiLiuLiuZhou_JPS2022, Xu_nature14, ZhaoZhouXuZhao_JPS2023, zhaoteng2023}. 

Optimal shape design has {witnessed} significant achievements in computational fluid dynamics \cite{MP}, structural mechanics \cite{AllaireCMAME05,Bendsoe,BB}, acoustics \cite{Osher2}, and photonic crystals \cite{CD}, etc. Shape optimization has various applications in electrochemistry as well \cite{Deng2018,MitchellOrtiz,Onishi2019,Yaji2018,YoonPark}. For example, the design of vanadium redox flow batteries \cite{Yaji2018} was investigated for rechargeable batteries in the storage system of renewable energy resources aiming for low carbon emissions. A topology optimization framework was established for designing porous electrodes with a hierarchy of length scales for low-conductivity materials \cite{Roy2022}. The level set method \cite{Ishizuka2020} was used for the design of anodes placed in an electroplating bath to achieve uniform deposition thickness. The optimized meso-scale structure for the electrolyte-anode interfaces of solid oxide fuel cells showed distinguished features \cite{Onishi2019}. The topological design of surface electrode distribution over piezoelectric sensors/actuators attached to a thin-walled shell structure was investigated \cite{Zhang2014} to reduce the sound radiation in an unbounded acoustic domain. Numerical methods for seeking optimal shape design include moving asymptotes \cite{Tortorelli2001,Leon2020}, threshold dynamics method \cite{HChen2022,HChen2024}, phase field method \cite{JLXZ,LiYi2022,YuLi2023}, and moving morphable components \cite{GZZ}, etc. However, it is worth noting that the governing system for modeling shape design in electrochemistry has been much simplified in literature. For instance, the detailed description of spatio-temporal distribution of ionic concentrations has been neglected and the governing system involves only the Poisson's equation for the chemical potential; e.g., \cite{Ishizuka2020,Onishi2019,Roy2022}. 

Ion transport in supercapacitors under electric fields has been successfully described by the well-known Poisson--Nernst--Planck (PNP) theory~\cite{WangPilon_JPCC13,YangJanssenLianRoij_jcp_22}. In such a theory, the electrodiffusion of ions {is} modeled by the Nernst--Planck equations under the gradient of electric potential, which in turn is governed by the Poisson's equation with charge sources stemming from mobile ions. In equilibrium, {the} ionic distribution in EDLs, in which charges are mainly stored in supercapacitors, can be described by the steady-state PNP system. Empirically, the charge storage, or energy density, can be enhanced with high electrode-electrolyte interface area to volume ratio, as larger electrode-electrolyte surface area allows more EDLs~\cite{Huang:ACIE:08a,Lian_PhysRevLett_2020,WVP:EA:11}. 
Sensitivity analysis was performed for the PNP system's solution {with respect to} (w.r.t.) physical coefficients \cite{Dione2019}. To the best of our knowledge, however, sensitivity analysis of the PNP system's solution w.r.t. geometric perturbations has not been studied for shape optimization in literature. In this work, we propose a shape optimization approach to maximize charge storage in supercapacitors through optimizing the electrode-electrolyte interface morphology. The total charge in a supercapacitor {is optimized} as an objective under the constraint of the steady-state PNP system, which has not been fully studied in previous works on optimal shape design in electrochemistry. 


For the sensitivity analysis of shape optimization, the \emph{Eulerian derivative} of a \emph{shape functional} can be derived via the velocity method \cite{SokolowskiZolesio1992}. The \emph{Eulerian derivative} is typically expressed in a boundary or domain formulation \cite{LS} (see, e.g., \cite{LZS,Zhang2024}). The former, characterized by the Hadamard-Zol'{e}sio structure theorem \cite{Defour}, is popular due to its concise formula. Variational numerical methods based on \emph{Eulerian derivative} can perform shape changes in shape design \cite{GLZ2022,Harbrecht,LiZhu2023,ZhuEig}. The derivation of \emph{Eulerian derivative} in boundary formulation involving extra introduction of the shape derivative thus requires more regularity on the computational domain. By the function space parametrization \cite{Defour}, the differentiability of the saddle point of a functional with respect to a real parameter { gives} explicit expressions of the Eulerian derivative in domain formulation. 

To numerically solve the PNP system, structure-preserving numerical methods that are able to preserve positivity of ionic concentrations and energy stability have been developed in literature \cite{Ding2019,HuHuang_NM20,LiuWangWiseYueZhou_2021,ShenXu_NM21}, especially for the case of large convection \cite{DingWangZhou_JCP2023,DingZhou_JCP2024}.  One of the most popular numerical schemes to solve steady-state PNP system is the Gummel fixed-point scheme \cite{Gummel1964,Markowich1986} which features excellent convergence properties even with { a} rough initial guess. By introducing the Slotboom transformation, the Gummel fixed-point scheme involves solving a semi-linear Poisson--Boltzmann type of equation for electric potential and the decoupled continuity equation for ionic concentration { alternately}. The inverse averaging on exponential coefficients for the decoupled continuity equation in finite element discretization has been proposed both in 2d \cite{Brezzi1989} and 3d \cite{averageJCP2022}, to overcome numerical instability caused by strong convection. The quality of the computational mesh is essential to obtain accurate approximations of ionic concentration and electric potential \cite{Brezzi1989,Prohl2009}. During shape evolution, a smooth descent direction calculated by a Cauchy-Riemann type gradient flow \cite{CR,LiZhu2023} { or} $H^1$ gradient flow \cite{ZhuEig} with uniform remeshing is used to keep the mesh shape regular and quasi-uniform.

The rest of the paper is organized as follows. In Section 2, we briefly recall the PNP system in a dimensionless form. Then we build the shape optimization model to maximize total charge storage {  subject to a} volume constraint. In Section 3, we utilize the framework of the velocity method  to perform shape sensitivity analysis by introducing a corresponding adjoint system.
In Section 4, a shape gradient algorithm is proposed by solving gradient flows to move the domain in a descent direction while preserving the quality of the mesh. In Section 5, we develop finite-element discretization of the state, the adjoint, and the gradient flow. In Section 6, various numerical experiments in 2d and 3d spaces are presented to demonstrate the effectiveness of the proposed algorithm.

\section{Shape optimization model}

\begin{figure}[htbp]
\centering
\begin{tikzpicture}[scale=3.5]
	\draw[thick,fill=gray!30!white] (1.0,1.0) -- (0.0,1.0) -- (0.0,0.0) -- (1.0,0.0) to[curve through =
	{(1.15,0.15) (1.2,0.2) (1.3,0.1) (1.4,0.05) (1.5,0.15) (1.55,0.1) (1.7,0.25) (1.4,0.35)  (1.8, 0.5) (1.5,0.65) (1.7,0.75) (1.55,0.9)  (1.5,0.85) (1.4,0.95) (1.3,0.9) (1.2,0.8) (1.15,0.85)}] (1.0,1.0);
	\draw[thick,dashed] (1.0,0) -- (1.85,0) -- (1.85,1.0) -- (1.0,1.0);
 	\draw[thick,dashed] (1.0,1.0) -- (1.0,0);
	\draw (0.5,0.5) node[thick, scale=1.25] {$\Omega_1$};
	\draw (1.35,0.5) node[thick, scale=1.25] {$\Omega_2$};
	\draw (-0.1,0.5) node[thick, scale=1.] {$\Gamma_{\rm in}$};
	\draw (0.5,-0.1) node[thick, scale=1.] {$\Gamma_1$};
	\draw (0.5,1.1) node[thick, scale=1.] {$\Gamma_1$};
	\draw (1.78,0.7) node[thick, scale=1.] {$\Gamma_2$};
\draw node[circle,draw][scale=0.5] at (0.75,0.25) {\textbf{+}};
\draw node[circle,draw][scale=0.5] at (0.25,0.3){\textbf{+}};
\draw node[circle,draw][scale=0.5] at (0.25,0.7){\textbf{+}};
\draw node[circle,draw][scale=0.5] at (1.4,0.7){\textbf{+}};
\draw node[circle,draw][scale=0.5] at (1.6,0.48){\textbf{+}};
\draw node[circle,draw][scale=0.5] at (1.6,0.2){\textbf{+}};
\draw node[circle,draw][scale=0.75] at (0.5,0.25) {\textbf{-}};
\draw node[circle,draw][scale=0.75] at (1.3,0.3) {\textbf{-}};
\end{tikzpicture}
\caption{Schematic illustration of a supercapacitor model with binary electrolytes. The bulk electrolyte domain $\Omega_1$ with boundaries $\Gamma_1$ and $\Gamma_{\text{in}}$ interconnects the domain $\Omega_2$ with a design boundary $\Gamma_2$ that is the electrolyte-electrode interface. }
\label{fig1}
\end{figure}

Let $\Omega \subset \mathbb{R}^d\ (d= 2,3)$ be an open bounded domain with Lipschitz boundary $\partial \Omega$. The domain $\Omega$ under consideration consists of two subdomains: $\overline{\Omega}= \overline{\Omega_1}\cup \overline{\Omega_2}$ with  disjoint boundaries satisfying $\partial\Omega = \overline{\Gamma_{\rm in}}\cup \overline{\Gamma_1}\cup \overline{\Gamma_2}$; cf. Fig.~\ref{fig1}.  
Denote by $\phi:\Omega\rightarrow\mathbb{R}$ the electric potential function and $c_i:\Omega\rightarrow\mathbb{R}\ (i=1,2,\cdots, N)$ the ionic concentration of the $i$-th species with valence $z_i\in \mathbb{R}$, where $N\in \mathbb{N}^+$ denotes the number of ionic species. Consider a boundary value problem of a steady-state Poisson--Nernst--Planck (PNP) system in a dimensionless form ~\cite{WangPilon_JPCC13,YangJanssenLianRoij_jcp_22}
\begin{equation}\label{PNP}
\left\{
\begin{aligned}
& - \nabla \cdot (\nabla c_i + z_i c_i \nabla \phi)=0,\ i=1,2,\cdots,N,  \quad &&{\rm in}\ \Omega,\\
& - \epsilon \Delta \phi = \bm z^{\rm T}\bm c, \quad &&{\rm in}\ \Omega, \\
&\bm c = \bm c^\infty, &&{\rm on}\  \Gamma_{\rm in},\\
& \frac{\partial c_i}{\partial \bm n} + z_ic_i \frac{\partial\phi}{\partial \bm n} = 0,\ i=1,2,\cdots,N, &&{\rm on}\ \Gamma_1 \cup\Gamma_2,\\
&\frac{\partial \phi}{\partial \bm n} = 0, &&{\rm on}\  \Gamma_{\rm in}\cup \Gamma_{1},\\ &\phi=g, &&{\rm on}\  \Gamma_2,
\end{aligned}\right.
\end{equation}
where $\bm c = [c_1,c_2,\cdots,c_N]^{\rm T}$, $\bm c^{\infty} = [c_1^{\infty}, c_2^{\infty}, \cdots,c_N^{\infty}]^{\rm T}$, $\bm z = [z_1,z_2,\cdots,z_N]^{\rm T}$, $\epsilon>0$ is a dimensionless rescaled dielectric constant, and $\bm{n}$ is an outward unit normal on the boundary. Here no-flux boundary conditions are imposed on $\Gamma_1$ and $\Gamma_2$. The Dirichlet boundary conditions on $\Gamma_{\rm in}$ are imposed for the ionic concentrations with  $c_i^\infty: \Gamma_{\rm in} \rightarrow \mathbb{R}^{+}$ being the bulk concentrations to describe that the supercapacitor is connected to an ionic reservoir with constant concentrations. 
For the electric potential, homogeneous Neumann boundary conditions are imposed on $\Gamma_{\rm in}\cup \Gamma_1$ and Dirichlet boundary { condition is} imposed on $\Gamma_2$ with given boundary data $g\in H^1(\mathbb{R}^d)$. {  For metal electrodes, it is reasonable to assume that the boundary data $g$ of the electric potential is a uniform constant.}

With ionic concentrations governed by the steady-state PNP system~\eqref{PNP}, let the integral $\int_{\Omega} \bm z^{\rm T} \bm c \dx$ represent the total net charge stored in a supercapacitor. It is of practical interest for electrochemical energy storage devices to maximize the total net charge~\cite{ BazantPRE,Simon_NatEnergy2016,Lian_PhysRevLett_2020}. We here consider fixed boundaries $\Gamma_{\rm in}$ and $\Gamma_1$ for the bulk domain and optimize the electrolyte-electrode interface $\Gamma_2$, at which the formed electric double layers account for the main net charge storage.
Generally, the shape optimization model is not well-posed without geometric constraints (e.g., volume or perimeter). To maximize charge storage, the following volume-constrained shape optimization model is considered: 
\begin{equation}\label{cost}
\min_{\Omega\in\mathcal{U}_{ad}} \mathcal{J}(\Omega, \bm c(\Omega)),
\end{equation}
where 
\begin{equation}\label{CostFun}
\mathcal{J}(\Omega, \bm c(\Omega))=\int_\Omega \sum_{i=1}^N j(c_i) \dx,
\end{equation}
with $j(c_i)=-z_i c_i$. Here $\bm{c}=\bm{c}(\Omega)$ is the solution to the system \eqref{PNP} and the admissible set reads
\begin{equation*}
\mathcal{U}_{ad}=\{\Omega\subset\mathbb{R}^d: \mathcal{P}_1(\Omega) = \mathcal{C}_1,\ \Omega\ \text{is Lipschitz},\ \Gamma_{\rm in}\cup {\Gamma_1}\ \text{is fixed} \},
\end{equation*}
where $\mathcal{P}_1(\Omega)$ is the geometric Lebesgue measure of $\Omega$ and $\mathcal{C}_1>0$ is a prescribed number. 
{  For the sake of readers' convenience, Table \ref{tab:my_label} lists all the technical parameters used in the PNP system and the following Algorithm \ref{alg1}.}

\begin{table}[htbp]
    \centering
\begin{tabular}{|c|c|}
\hline
 Terminologies & \textbf{Nomenclature} \\
\hline $c_i$ & ionic
concentration\\
\hline $\phi$ & electric potential \\
\hline $z_i$ & valence of the $i$-th species \\
\hline $\epsilon$ & dimensionless rescaled
dielectric constant\\
\hline $t$ & time variable\\
\hline $T_t$ & domain mapping\\
\hline $\bm{\theta}$ & velocity field\\
\hline $s_i$ & adjoint ionic concentration\\
\hline $\psi$ & adjoint electric potential\\
\hline
$l$ & Lagrange multiplier \\
\hline
$\beta$ & penalty parameter for volume constraint \\
\hline
$\mathcal{C}_1$ & target volume  \\
\hline
$\epsilon_0$ &  diffusion parameter in $H^1$ gradient flow \\
\hline
$\alpha$ & penalty parameter in C-T gradient flow\\
\hline
$\tau$ & stopping tolerance for solving the PNP system \\
\hline
$\delta_k$ & step size for mesh movement \\
\hline
$M$ & penalty parameter for descent direction \\
\hline
$\gamma$ & penalty parameter for perimeter control \\
\hline
$\epsilon$ & rescaled dielectric constant \\
\hline 
$c_i^\infty$ & bulk concentrations \\
\hline  
$g$ & electric potential on boundaries \\
\hline
\end{tabular}
    \caption{{ Terminologies concerning the PNP system and gradient flow in the Algorithm \ref{alg1}}.}
    \label{tab:my_label}
\end{table}

\section{Shape sensitivity analysis}
In this section, we perform \emph{shape sensitivity analysis} using the velocity method \cite{Defour,SokolowskiZolesio1992} to derive Eulerian derivatives in both volumetric and boundary expressions. Standard notations for Sobolev spaces are used in the following. Introduce scalar sets $H^1_{\Gamma_2, d}(\Omega):=\{w\in H^1(\Omega) | w = g\ {\rm on }\ \Gamma_2 \}$ for the electric potential and $H^1_{\Gamma_{\rm in}, i}(\Omega):=\{w\in H^1(\Omega) | w = c_i^\infty\ {\rm on }\ \Gamma_{\rm in} \}$ ($i=1,2,\cdots,N)$ for each ionic concentration function $c_i$, and $\textbf{H}^1_{\Gamma_{\rm in}, d}(\Omega):= H^1_{\Gamma_{\rm in}, 1}(\Omega)\times H^1_{\Gamma_{\rm in}, 2}(\Omega)\times\cdots\times H^1_{\Gamma_{\rm in}, N}(\Omega)$ for vectors of concentration functions. Let $H^1_{\Gamma_{\rm in}, 0}(\Omega):=\{w\in H^1(\Omega)| w=0\ {\rm on}\ \Gamma_{\rm in} \}$ and $H^1_{\Gamma_2,0}(\Omega):=\{w\in H^1(\Omega)| w=0\ {\rm on}\ \Gamma_2 \}$ be two subspaces of $H^1(\Omega)$. Introduce a Sobolev space $\mathbf{H}^1(\Omega):= [H^1(\Omega)]^N$ and its subspace $\textbf{H}^1_{\Gamma_{\rm in},0}(\Omega):=\{\bm w\in \textbf{H}^1(\Omega)| \bm w =\bm 0\ {\rm on}\ \Gamma_{\rm in} \}$.

Let $\vec{\theta}\in C([0,\tau_0);\mathcal{D}^1(\mathbb{R}^d,\mathbb{R}^d))$ be a velocity field, with a sufficiently small { $\tau_0>0$}, characterizing morphology deformation over the whole domain, where $\mathcal{D}^1(\mathbb{R}^d,\mathbb{R}^d)$ denotes the space of all continuously differentiable functions with compact support in $\mathbb{R}^d$. For each { $0<t<\tau_0$}, we denote $\vec{\theta}(t)(\bm x):=\vec{\theta}(t,\bm x)\in\mathcal{D}^1(\mathbb{R}^d,\mathbb{R}^d)$. { For a given} $\bm X\in \mathbb{R}^d$, it generates a family of transformations
\begin{equation*}
T_t(\vec{\theta})(\bm X)=\bm x(t,\bm X),
\end{equation*}
through the following dynamical system:
\begin{equation*}\label{velocityequ}
\left\{
\begin{aligned}
\frac{{\rm d}\bm x}{{\rm d}t}(t,\bm X) &=\vec{\theta}(t,\bm x(t,\bm X)),\ && t\in (0,\tau_0), \\
\bm x(0,\bm X) &=\bm X,\ && t=0.
\end{aligned}
\right.
\end{equation*}
The transformed domain $T_t(\vec{\theta})(\Omega)$ { is referred to as} $\Omega_t$. We introduce the following definitions { related to} shape calculus.
\begin{Definition}\cite{Defour,HaslingerMakinen2003,SokolowskiZolesio1992}
Let $\mathcal{J}(\Omega, \bm c(\Omega))$ be a {shape functional} $\mathcal{J}(\cdot, \bm c(\cdot)):\Omega\mapsto\mathbb{R}$ and  $\vec{\theta}\in C([0,\tau_0);\mathcal{D}^1(\mathbb{R}^d,\mathbb{R}^d))$ be a velocity field. 
\begin{itemize}
    \item The {Eulerian derivative} of $\mathcal{J}(\Omega, \bm c(\Omega))$ at $\Omega$ in the direction $\vec{\theta}$ is defined by 
	\begin{equation*}\label{eulerder}
	{\rm d}\mathcal{J}(\Omega;\vec{\theta}):=\lim_{t\rightarrow 0^{+}} \frac{\mathcal{J}(\Omega_t, \bm c(\Omega_t))-\mathcal{J}(\Omega, \bm c(\Omega))}{t}.
	\end{equation*}
    Denote by ${\rm d}_{\mathcal{V}}\mathcal{J}(\Omega;\vec{\theta})$ and  ${\rm d}_{\mathcal{S}}\mathcal{J}(\Omega;\vec{\theta})$ the volume expression and the surface expression of the Eulerian derivative, respectively.
    
    \item The {material derivative} of a scalar state variable $w\in H^1(\Omega)$ in the direction $\vec{\theta}$ is defined by 
\begin{equation*}
\dot{w}(\Omega ; \vec{\theta}):=\lim _{t \rightarrow 0^{+}} \frac{w\left(\Omega_{t}\right) \circ T_{t}(\Omega)-w(\Omega)}{t}.
\end{equation*}

\item The {shape derivative} of $w$ in the direction $\vec{\theta}$ is 
defined by
\begin{equation*}
w^\prime (\Omega;\vec{\theta}) := \dot{w} (\Omega;\vec{\theta})-\nabla w \cdot \vec{\theta}(0).
\end{equation*}
\end{itemize}
For a vector state variable $\bm w:= [w_1, w_2,\cdots, w_N]^{\rm T}$, its material derivative, denoted by $\dot{\bm w}$, is defined by the material derivative of each component, i.e., $\dot{\bm w}:=[\dot{w}_1, \dot{w}_2,\cdots, \dot{w}_N]^{\rm T}$. Similarly, for a vector state variable $\bm w:= [w_1, w_2,\cdots, w_N]^{\rm T}$, its shape derivative is defined by $${\bm w}^\prime:=[w_1^\prime, w_2^\prime,\cdots, w_N^\prime]^{\rm T}.$$
\end{Definition}
Assume that the boundary is sufficiently smooth enough. Then tangential gradient and divergence operators can be defined as follows:
\begin{Definition}
Let $\Omega$ be of $C^2$ class with the unit outward normal vector $\bm n$ on $\partial\Omega$. Suppose that $U\subset \mathbb{R}^d$ is an open bounded set such that $\Omega\subset\subset U$.
\begin{itemize}
    \item The {tangential divergence} of  a vector variable $\bm w\in C^1(U;\mathbb{R}^d)$ along the boundary { $\partial\Omega$} is defined by
\begin{equation*}
{\rm div}_{\Gamma}\,\bm w:= {\rm div}\,\bm w- \bm n^{\rm T} {\rm D}\bm w \bm n,
\end{equation*}
where ${\rm div}(\cdot)$ and $\rm D(\cdot)$ are the standard divergence and vector gradient operators defined { over} $\textbf{H}^1(U)$, respectively.

\item The {tangential gradient} $\nabla_{\Gamma}$ of a scalar variable $w\in C^1 (U;\mathbb{R})$ on the boundary { $\partial\Omega$} is defined by
\begin{equation*}
\nabla_{\Gamma} w := \nabla w-\frac{\partial w}{\partial \bm n}\bm n.
\end{equation*}
\end{itemize}
\end{Definition}

\begin{Proposition}\cite{HaslingerMakinen2003} For a deformation field $\vec{\theta}=[\theta_1, \theta_2,\cdots,\theta_d]^{\rm T}\in  {W}^{1,\infty}$$(\Omega)^d$, the following results hold
 \begin{equation*}
 \begin{aligned}
 &(1) \left.{\rm D} T_{t}\right|_{t=0}=\mathrm{Id},\quad &&(2) \left.\frac{{\rm d}}{{\rm d} t} T_{t}\right|_{t=0}=\vec{\theta},\\
&(3) \left.\frac{{\rm d}}{{\rm d} t} {\rm D} T_{t}\right|_{t=0}={\rm D} {\vec{\theta}}=\left(\frac{\partial {\theta}_{i}}{\partial x_{j}}\right)_{i, j=1}^{{\rm d}},
\quad &&(4) \left.\frac{{\rm d}}{{\rm d} t}{\rm D} T_{t}^{\mathrm{T}}\right|_{t=0}={\rm D} {\vec{\theta}}^{\mathrm{T}},\\
&(5) \left.\frac{{\rm d}}{{\rm d} t}\left({\rm D} T_{t}^{-1}\right)\right|_{t=0}=-{\rm D} {\vec{\theta}},
\quad &&(6) \left.\frac{{\rm d}}{{\rm d} t} J_{t}\right|_{t=0}=\operatorname{div} {\vec{\theta}},
\end{aligned}
\end{equation*}
where ${\rm Id}$ is an identity matrix and $J_t=\det{({\rm D}T_t)}$ denotes the determinant of the Jacobian.
\end{Proposition}
In shape optimization, apart from geometric constraints, { the problem} generally involves physical state constraints. Treating the state system as an equality constraint and using the Lagrange multiplier method typically { yields} dual variables that are solutions of an adjoint system. {   For ease of presentation of shape sensitivity analysis, we assume that the potential $\phi$ has zero boundary conditions on $\Gamma_2$. Otherwise, one can replace the original potential function with $\phi-g$ in \eqref{PNP} to obtain the homogeneous Dirichlet boundary condition on $\Gamma_2$, since the Dirichlet data $g$ is a constant defined over the whole $\mathbb{R}^d$. }

\begin{Lemma}\label{adjLemma}
Suppose that {  $(\bm c,\phi)\in \textbf{H}^1_{\Gamma_{\rm in},d}(\Omega)\times H^1_{\Gamma_2,0}(\Omega)$} is the solution of PNP system (\ref{PNP}). The adjoint system to the optimization problem (\ref{cost}) constrained by (\ref{PNP}) is to find a weak solution $(\bm s,\psi)\in \textbf{H}^1_{\Gamma_{\rm in},0}(\Omega)\times H^1_{\Gamma_2,0}(\Omega)$ 
satisfying the following boundary value problem
\begin{equation}\label{adjPro}\left\{
\begin{aligned}
&-\Delta s_i+ z_i\nabla\phi\cdot \nabla s_i+z_i \psi = j^\prime(c_i),\ i=1,2,\cdots,N, \quad &&{\rm in}\ \Omega,\\
&-\epsilon\Delta \psi = -\sum_{i=1}^N \nabla\cdot(z_i c_i\nabla s_i), \quad &&{\rm in}\ \Omega,\\
&\bm s =0, && {\rm on}\ \Gamma_{\rm in},\\
&\frac{\partial s_i}{\partial \bm n}=0, && {\rm on}\ \Gamma_1\cup \Gamma_2,\\
&\psi = 0, && {\rm on}\ \Gamma_2,\\
&\epsilon \frac{\partial \psi}{\partial \bm n} =\sum_{i=1}^N z_ic_i \frac{\partial s_i}{\partial \bm n}, && {\rm on}\ \Gamma_{\rm in}\cup\Gamma_1.
\end{aligned}\right.
\end{equation}
\end{Lemma}
\begin{proof}
By the weak formulation of the PNP system (\ref{PNP}), we define a functional 
\begin{equation}\label{adjLemmaS1}
\begin{aligned}
\mathcal{G}(\Omega, \bm c, \phi, \bm s,\psi):= \sum_{i=1}^N \int_\Omega (\nabla c_i +z_i c_i \nabla \phi)\cdot \nabla s_i\dx -\int_\Omega \epsilon\nabla\phi\cdot \nabla \psi- \bm z^{\rm T} \bm c \psi\dx,
\end{aligned}
\end{equation}
where the test functions $(\bm s,\psi)\in \textbf{H}^1_{\Gamma_{\rm in},0}(\Omega)\times H^1_{\Gamma_2,0}(\Omega)$. Then we introduce a Lagrangian
\begin{equation}\label{adjLemmaS2}
\mathcal{L}(\Omega,\bm c,\phi, \bm s,\psi):= \mathcal{J}(\Omega, \bm c(\Omega)) - \mathcal{G}(\Omega, \bm c,\phi, \bm s,\psi).
\end{equation}
A saddle point of $\mathcal{L}$ is characterized by
\begin{eqnarray}
    &\frac{\partial \mathcal{L}}{\partial \tilde{s}_i}(\delta s_i)=\frac{\partial \mathcal{L}}{\partial \tilde{\psi}}(\delta \psi)=0\quad \forall \,  \delta\bm s =[\delta s_1, \delta s_2,\cdots, \delta s_N] \in \textbf{H}^1_{\Gamma_{\rm in},0}(\Omega)\quad\forall \delta \psi \in H^1_{\Gamma_2,0}(\Omega),\label{KKT1}
    \\
    &\frac{\partial \mathcal{L}}{\partial \tilde{c_i}}(\delta c_i)=\frac{\partial \mathcal{L}}{\partial \tilde{\phi}}(\delta \phi)=0\quad \forall \, \delta \bm c=[\delta c_1, \delta c_2,\cdots, \delta c_N]\in \textbf{H}^1_{\Gamma_{\rm in},0}(\Omega)\quad\forall \delta \phi \in H^1_{\Gamma_2,0}(\Omega).\label{KKT2}
\end{eqnarray}
Eq.~\eqref{KKT1} corresponds to the weak formulation of PNP system (\ref{PNP}) {   with $g=0$}, and Eq.~(\ref{KKT2}) implies an adjoint state system that we shall derive in the following. 

The Fr\'{e}chet derivative of the Lagrangian (\ref{adjLemmaS2}) with respect to each concentration yields the first-order necessary optimality condition
\begin{equation*}
\begin{aligned}
0 = &\frac{\partial \mathcal{L}}{\partial \tilde{c_i}}(\delta c_i)\\
=& \int_\Omega -j^\prime(c_i)\delta c_i + (\nabla \delta c_i +z_i \delta c_i \nabla \phi)\cdot \nabla s_i\dx +\int_\Omega  z_i \delta c_i \psi\dx \\
=& \int_\Omega -j^\prime(c_i)\delta c_i -\Delta s_i \delta c_i+z_i\nabla\phi\cdot\nabla s_i \delta c_i +  z_i \psi\delta c_i \dx +\int_{\partial\Omega}\frac{\partial s_i}{\partial \bm n}\delta c_i \ds \quad \forall \delta c_i \in H^1_{\Gamma_{\rm in},0}(\Omega),
\end{aligned}
\end{equation*}
which further implies that
\begin{equation*}\left\{
\begin{aligned}
&-\Delta s_i+ z_i\nabla\phi\cdot \nabla s_i+z_i \psi = j^\prime(c_i),\quad &&{\rm in}\ \Omega,\\
&\frac{\partial s_i}{\partial \bm n}=0, && {\rm on}\ \Gamma_1\cup\Gamma_2,\\
&s_i=0,&&{\rm on}\ \Gamma_{\rm in}.
\end{aligned}\right.
\end{equation*}
The Fr\'{e}chet derivative of the Lagrangian (\ref{adjLemmaS2}) w.r.t. the electric potential leads to the optimality condition
\begin{equation*}
\begin{aligned}
0 =& \frac{\partial \mathcal{L}}{\partial \tilde{\phi}}(\delta \phi)\\
=&\sum_{i=1}^N \int_\Omega z_i c_i \nabla \delta \phi \cdot \nabla s_i\dx- \int_\Omega \epsilon \nabla\delta\phi\cdot \nabla \psi\dx,\\
=& \int_\Omega -\sum_{i=1}^N  \nabla\cdot (z_i c_i\nabla s_i)\delta \phi +\epsilon\Delta\psi \delta\phi \dx + \int_{\Gamma_{\rm in}\cup \Gamma_1}\bigg( \sum_{i=1}^N z_ic_i \frac{\partial s_i}{\partial \bm n}-\epsilon \frac{\partial \psi}{\partial \bm n}  \bigg)\delta\phi \ds \quad \forall \delta \phi \in H^1_{\Gamma_2,0}(\Omega),
\end{aligned}
\end{equation*}
which further yields that
\begin{equation*}\left\{
\begin{aligned}
&-\epsilon\Delta \psi = -\sum_{i=1}^N \nabla\cdot(z_i c_i\nabla s_i), \quad &&{\rm in}\ \Omega,\\
&\psi = 0, && {\rm on}\ \Gamma_2,\\
&\sum_{i=1}^N z_ic_i \frac{\partial s_i}{\partial \bm n}-\epsilon \frac{\partial \psi}{\partial \bm n} = 0, && {\rm on}\ \Gamma_{\rm in}\cup \Gamma_1.
\end{aligned}\right.
\end{equation*}
This completes the derivation of the adjoint system (\ref{adjPro}).
\end{proof}

The shape sensitivity analysis is based on the framework of the Lagrangian method \cite{Defour} by introducing a Lagrangian to convert the original constrained optimization problem into a saddle point problem. We introduce the following \emph{function space parametrization} method; cf.~\cite[Section 5.3]{Defour}. For both the state and adjoint state functions defined on the perturbed domain $\Omega_t$, we have the following parametrization
\begin{equation*}
H^1 (\Omega_t) = \{\phi\circ T_t^{-1}: \phi\in H^1(\Omega) \},
\end{equation*}
where $t$ is small enough such that $T_t$ and $T_t^{-1}$ are diffeomorphisms. Lagrangian methods in shape optimization allow { us to} compute the Eulerian derivative of the shape functional depending on the solution of the governing system without the need to calculate the material derivative of state variables.  For any vector functions $\bm w = [w_1,w_2,\cdots,w_N]^{\rm T}$ and $\bm y = [y_1,y_2,\cdots,y_N]^{\rm T}$, the Frobenius inner products for gradients and tangential gradients { are} defined by ${\rm D}\bm w : {\rm D} \bm y:= \sum_{i=1}^N \nabla w_i\cdot \nabla y_i$ and ${\rm D}_\Gamma\bm w : {\rm D}_\Gamma \bm y:= \sum_{i=1}^N \nabla_\Gamma w_i\cdot \nabla_\Gamma y_i$, respectively.
\begin{Theorem}\label{dJVolThm}
Let the  $\Omega$ be a bounded domain with a Lipschitz continuous boundary. Suppose that  {  $(\bm c,\phi)\in \textbf{H}^1_{\Gamma_{\rm in},d}(\Omega)\times H^1_{\Gamma_2,0}(\Omega)$} is the solution of (\ref{PNP}) and $(\bm s,\psi)\in \textbf{H}^1_{\Gamma_{\rm in},0}(\Omega)\times H^1_{\Gamma_2,0}(\Omega)$ is the solution of the adjoint problem (\ref{adjPro}). Then the Eulerian derivative of the shape functional { (\ref{CostFun})} in domain expression reads
\begin{equation}\label{dJVol}
\begin{aligned}
{\rm d}_\mathcal{V}\mathcal{J}(\Omega,\vec{\theta})&= \int_\Omega (-\bm z^{\rm T} \bm c \psi-\bm z^{\rm T} \bm c) {\rm div}\vec{\theta}+M(\vec{\theta}){\rm D} \bm c : {\rm D}\bm s +  M(\vec{\theta}) {\rm D}\Phi : ({\rm D} \bm s \bm Z\bm C)+ \epsilon M(\vec{\theta})\nabla \phi\cdot\nabla\psi\dx,
\end{aligned}
\end{equation}
where the matrix function $M(\vec{\theta}):={\rm div}\vec{\theta}{\rm Id}-{\rm D}\vec{\theta}^{\rm T}-{\rm D}\vec{\theta}$
and
\begin{equation}\label{ZCPHI}
\begin{aligned}
    \bm Z:= {\rm diag}[z_1,z_2,\cdots,z_N],\quad \bm C:= {\rm diag}[c_1,c_2,\cdots,c_N],\quad \Phi:= [\phi,\phi,\cdots,\phi]^{\rm T}\in [H_{\Gamma_2,d}^1(\Omega)]^N,
\end{aligned}
\end{equation}
with ${\rm diag}$ denoting a diagonal matrix formed by a vector and ${\rm Id}$ being the identity matrix.
\end{Theorem}
\begin{proof}
{  We formally present the derivation of Eulerian derivative below. Define} the Lagrangian with Lagrange multipliers:
\begin{equation*}
    \mathcal{L}(\Omega, \bm c, \phi, \bm s, \psi) = \sum_{i=1}^N \int_{\Omega}  j(c_i)\dx+\sum_{i=1}^N \int_\Omega (\nabla c_i +z_i c_i \nabla \phi)\cdot \nabla s_i\dx +\int_\Omega \epsilon\nabla\phi\cdot \nabla \psi- \bm z^{\rm T} \bm c \psi\dx,
\end{equation*}
where $(\bm c, \phi)$ is the solution of \eqref{PNP} and $(\bm s, \psi)$ are adjoint variables governed by (\ref{adjPro}). Then the Lagrangian $\mathcal{L}$ on the perturbed domain $\Omega_t$ reads
\begin{equation*}
\begin{aligned}
    \mathcal{L}(\Omega_t, \bm c_t, \phi_t, \bm s_t, \psi_t) = &\sum_{i=1}^N \int_{\Omega_t}  j(c_{i,t})\dx+\sum_{i=1}^N \int_{\Omega_t} (\nabla c_{i,t} +z_i c_{i,t} \nabla \phi_t)\cdot \nabla s_{i,t}\dx +\int_{\Omega_t} \epsilon\nabla\phi_t\cdot \nabla \psi_t- \sum_{i=1}^N z_i c_{i,t} \psi_t\dx,
    \end{aligned}
\end{equation*}
where $c_{i,t}=c_i \circ T^{-1}_t$, $s_{i,t}=s_i \circ T^{-1}_t$, $\phi_{t}=\phi \circ T^{-1}_t$, and $\psi_{t}=\psi \circ T^{-1}_t$ with $c_{i}\in H^1_{\Gamma_{\rm in}, i}(\Omega), s_{i}\in H^1_{\Gamma_{\rm in}, 0}(\Omega)$, {  $\phi\in H^1_{\Gamma_2, 0}(\Omega)$}, 
 and $\psi \in H^1_{\Gamma_2, 0}(\Omega)$. 
 We next define
\begin{equation*}
\tilde{\mathcal{L}}(t,\bm c,\phi,\bm s,\psi) = \mathcal{L}(\Omega_t, \bm c\circ T^{-1}_t, \phi\circ T^{-1}_t, \bm s\circ T^{-1}_t, \psi\circ T^{-1}_t).
\end{equation*}
By definition, $\tilde{\mathcal{L}}(t, \cdot)$ is given by 
\begin{equation}\label{GOmegat}
\begin{aligned}
 \tilde{\mathcal{L}}(t,\bm c,\phi,\bm s,\psi)= & \sum_{i=1}^N \int_{\Omega_t}  -z_i (c_i\circ T^{-1}_t) \dx+\sum_{i=1}^N \int_{\Omega_t} \big[\nabla (c_{i}\circ T^{-1}_t) +z_i (c_{i}\circ T^{-1}_t) \nabla (\phi\circ T^{-1}_t)\big]\cdot \nabla (s_{i}\circ T^{-1}_t)\dx \\
 &+\int_{\Omega_t} \epsilon\nabla(\phi\circ T^{-1}_t)\cdot \nabla (\psi\circ T^{-1}_t)- \sum_{i=1}^N z_i (c_{i}\circ T^{-1}_t) (\psi\circ T^{-1}_t)\dx.
 \end{aligned}
\end{equation}
Rewriting this functional on the fixed domain $\Omega$ by using the transformation $T_t$ gives
 \begin{equation*}
\begin{aligned}
 \tilde{\mathcal{L}}(t,\bm c,\phi,\bm s,\psi)= &  \int_{\Omega}  -\bm z^{\rm T} \bm c J_t \dx+\sum_{i=1}^N \int_{\Omega}\mathcal{A}(t)\nabla c_i\cdot \nabla s_i +z_i \mathcal{A}(t) c_i \nabla \phi \cdot \nabla s_i \dx \\
 &+\int_{\Omega} \epsilon\mathcal{A}(t) \nabla \phi\cdot \nabla \psi- \bm z^{\rm T} \bm c \psi J_t\dx,
 \end{aligned}
\end{equation*}
where $\mathcal{A}(t)=J_t [\md T_t]^{-1}[\md T_t]^{-{\rm T}}$. Taking derivative of $\tilde{\mathcal{L}}(t,\cdot)$ with respect to $t$ yields
\begin{equation*}
\begin{aligned}
 \partial_t \tilde{\mathcal{L}}(t,\bm c,\phi,\bm s,\psi)= &  \int_{\Omega}  -\bm z^{\rm T} \bm c J_t^\prime \dx+\sum_{i=1}^N \int_{\Omega}\mathcal{A}^\prime(t)\nabla c_i\cdot \nabla s_i +z_i \mathcal{A}^\prime(t) c_i \nabla \phi \cdot \nabla s_i \dx \\
 &+\int_{\Omega} \epsilon\mathcal{A}^\prime(t) \nabla \phi\cdot \nabla \psi- \bm z^{\rm T} \bm c \psi J_t^\prime\dx.
 \end{aligned}
\end{equation*}
Denoting $\mathcal{Q}(t):= ({\rm D}T_t)^{-1}$ and using properties that
\begin{equation*}
\begin{aligned}
&\mathcal{Q}^\prime(t) = -\mathcal{Q}(t) {\rm D}\bm \theta(t) \circ T_t,\\
&\mathcal{A}^\prime(t) \bm \tau: \bm \sigma=[\bm \tau \mathcal{Q}^\prime(t)]:[\bm \sigma \mathcal{Q}(t)]+[\bm \tau \mathcal{Q}(t)]:[\bm \sigma \mathcal{Q}^\prime(t)],\\
& J^\prime_t = {\rm div} \bm \theta(t),
\end{aligned}
\end{equation*}
we obtain
\begin{equation*}
\mathcal{A}^\prime(0) \bm \tau: \bm \sigma=-\left[\left({\rm D}\bm \tau {\rm D}\bm \theta(0): {\rm D}\bm \sigma)+ ({\rm D}\bm \tau : {\rm D}\bm \sigma {\rm D}\bm \theta(0)\right)\right].
\end{equation*}
By formally using the Correa-Seeger Theorem \cite{{Defour}}, we obtain the Eulerian derivative
\begin{equation*}
\begin{aligned}
{\rm d}_\mathcal{V}\mathcal{J}(\Omega,\vec{\theta})=&\lim_{t\rightarrow 0^+}\partial_t \tilde{\mathcal{L}}(t,\bm c,\phi,\bm s,\psi)\\
=&\sum_{i=1}^N \int_\Omega -z_i c_i {\rm div}\vec{\theta}+M(\vec{\theta})\nabla c_i \cdot \nabla s_i + z_i c_i M(\vec{\theta}) \nabla\phi\cdot \nabla s_i \dx+\int_\Omega \epsilon M(\vec{\theta})\nabla \phi\cdot\nabla\psi - \bm z^{\rm T} \bm c \psi {\rm div}\vec{\theta}\dx,
\end{aligned}
\end{equation*}
where $\bm \theta:=\bm \theta(0)$ and $M(\vec{\theta})={\rm div}\vec{\theta}{\rm Id}-{\rm D}\vec{\theta}^{\rm T}-{\rm D}\vec{\theta}$. After matrix and vector rearrangements, the Eulerian derivative in domain expression is given by the matrix formulation (\ref{dJVol}).
\end{proof}

We next derive the shape derivative of $c_i$ and $\phi$ that are characterized by the following Lemma. 
\begin{Lemma}\label{shapeDerLemma}
Suppose that the boundary of $\Omega$ is of the class $C^2$ or a convex { polygon}. Let {  $(c_i,\phi)\in H^1_{\Gamma_{\rm in},i}(\Omega)\times H^1_{\Gamma_{\rm 2},0}(\Omega)$} be the solution of the PNP system (\ref{PNP}). Then the shape derivative $(c_i^\prime,\phi^\prime)$ satisfies the following problem 
\begin{equation}\label{shapeDerState}
\left\{
\begin{aligned}
&-\Delta c_i^\prime - \nabla\cdot (z_i c_i^\prime \nabla \phi+ z_i c_i \nabla \phi^\prime)=0,\ i=1,2,\cdots,N,\qquad && {\rm in}\ \Omega,\\
&-\epsilon\Delta \phi^\prime = \bm z^{\rm T} \bm c^\prime, && {\rm in}\ \Omega,\\
&(\nabla c_i^\prime + z_i c_i^\prime \nabla \phi+ z_i c_i \nabla \phi^\prime)\cdot \bm n = {\rm div}_\Gamma \left[{\theta}_n (\nabla c_i + z_i c_i \nabla \phi)\right],\ i=1,2,\cdots,N, &&{\rm on}\ \Gamma_1\cup\Gamma_2, \\
& \bm c^\prime = \frac{\partial(\bm c^{\infty}- \bm c)}{\partial \bm n}\theta_n, &&{\rm on}\ \Gamma_{\rm in},\\
&\epsilon \frac{\partial \phi^{\prime}}{\partial \bm n}= \bm z^{\rm T} \bm c\theta_n +\epsilon{\rm div}_{\Gamma}(\theta_n \nabla \phi), &&{\rm on}\ \Gamma_{\rm in}\cup\Gamma_1,\\
&{  \phi^\prime =- \frac{\partial \phi}{\partial \bm n} {\theta}_{n}}, \qquad &&{\rm on}\ \Gamma_2,
\end{aligned}\right.
\end{equation}
where $\bm c^\prime:=[c_1^\prime, c_2^\prime,\cdots,c_N^\prime]$ is the shape derivative of the vector function $\bm c$ and ${\theta}_n:=\bm\theta\cdot \bm n$.
\end{Lemma}
\begin{proof}
The weak formulation of the Nernst--Planck system in (\ref{PNP}) with boundary conditions on the perturbed domain $\Omega_t$ is to seek $c_{i,t} \in H^1_{\Gamma_{\rm in}, i}(\Omega_t)$ and { $\phi_t \in H^1_{\Gamma_2, 0}(\Omega_t)$} such that
\begin{equation*}
    \int_{\Omega_t} (\nabla c_{i,t} + z_i c_{i,t} \nabla \phi_t)\cdot \nabla v \dx = 0 \qquad \forall v \in H^1_{\Gamma_{\rm in}, 0}(\Omega_t).
\end{equation*}
For a sufficiently smooth function $F: [0,\tau_0)\times \mathbb{R}^d \rightarrow \mathbb{R}$, it follows from the Hadamard's formula that
\begin{equation*}
\frac{{\rm d}}{{\rm d}t}\int_{\Omega_t} F(t,\bm x) \dx \bigg|_{t =0} = \int_{\partial\Omega} F(0,\bm x)\bm \theta(0)\cdot\bm n\ds + \int_\Omega \frac{\partial F}{\partial t}(0,\bm x)\dx.
\end{equation*}
Then the shape derivative of the concentration function satisfies 
\begin{equation}\label{ss1}
    \int_\Omega (\nabla c_i^\prime + z_i c_i^\prime \nabla \phi+ z_i c_i \nabla \phi^\prime)\cdot \nabla v \dx + \int_{\partial\Omega} (\nabla c_i + z_i c_i \nabla \phi)\cdot \nabla v {\theta}_n \ds= 0\quad \forall v \in H^1_{\Gamma_{\rm in}, 0}(\Omega_t).
\end{equation}
Using the integration by parts yields that
\begin{equation*}
\begin{aligned}
0=&\int_\Omega -\Delta c_i^\prime v - \nabla\cdot (z_i c_i^\prime \nabla \phi+ z_i c_i \nabla \phi^\prime)v \dx+ \int_{\partial\Omega} \frac{\partial c_i^\prime}{\partial \bm n}v +(z_i c_i^\prime \nabla \phi+ z_i c_i \nabla \phi^\prime)\cdot \bm n v \ds\\
&+\int_{\partial\Omega} (\nabla c_i + z_i c_i \nabla \phi)\cdot \nabla v {\theta}_n\ds\\
=&\int_\Omega -\Delta c_i^\prime v - \nabla\cdot (z_i c_i^\prime \nabla \phi+ z_i c_i \nabla \phi^\prime)v \dx \\
&+\int_{\Gamma_1 \cup \Gamma_2}\frac{\partial c_i^\prime}{\partial \bm n}v +(z_i c_i^\prime \nabla \phi+ z_i c_i \nabla \phi^\prime)\cdot \bm n v +(\nabla c_i + z_i c_i \nabla \phi)\cdot \nabla v {\theta}_n \ds,
\end{aligned}
\end{equation*}
where $\theta_n=0$ and $v=0$ on $\Gamma_{\rm in}$ has been used. Taking the test function $v$ with compact support in $\Omega$ implies that
\begin{equation*}
    -\Delta c_i^\prime - \nabla\cdot (z_i c_i^\prime \nabla \phi+ z_i c_i \nabla \phi^\prime)=0\quad {\rm in}\ \Omega.
\end{equation*}
Taking the test function with a homogeneous Neumann condition $\frac{\partial v}{\partial \bm n}=0$ on $\Gamma_1\cup\Gamma_2$, we obtain 
\begin{equation*}
\int_{\Gamma_1 \cup \Gamma_2} \frac{\partial c_i^\prime}{\partial \bm n}v +(z_i c_i^\prime \nabla \phi+ z_i c_i \nabla \phi^\prime)\cdot \bm n v \ds+\int_{\Gamma_1 \cup \Gamma_2} (\nabla c_i + z_i c_i \nabla \phi)\cdot \nabla_\Gamma v {\theta}_n\ds = 0.
\end{equation*}
By the density property and tangential Green’s formula, we derive the following boundary condition
\begin{equation*}
(\nabla c_i^\prime + z_i c_i^\prime \nabla \phi+ z_i c_i \nabla \phi^\prime)\cdot \bm n = {\rm div}_\Gamma [{\theta}_n (\nabla c_i + z_i c_i \nabla \phi)]\quad {\rm on}\ \Gamma_1 \cup \Gamma_2.
\end{equation*}
 The material derivative of $c_i$ on $\Gamma_{\rm in}$ satisfies that $$\dot{c}_i=\dot{(c_i^\infty)}= \nabla c_i^\infty \cdot \bm \theta$$ 
 and
 \begin{equation*}
     \nabla_{\Gamma}(c_i - c_i^{\infty}) =\bm{0},
 \end{equation*}
implying $$c_i^\prime = \dot{c_i} - \nabla c_i \cdot \bm \theta = (\nabla c_i^\infty-\nabla c_i) \cdot \bm \theta=\frac{\partial(c_i^{\infty}-c_i)}{\partial \bm n}\theta_n.$$ 
Similarly, the weak formulation of the Poisson equation in (\ref{PNP}) for $\phi_t $ on the perturbed domain reads
\begin{equation*}
\int_{\Omega_t} \epsilon\nabla \phi_t \cdot \nabla v \dx = \int_{\Omega_t} \sum_{i=1}^N z_i c_{i, t} v \dx\quad \forall v \in H^1_{\Gamma_2,0}(\Omega_t).
\end{equation*}
Hence, utilizing the Hadamard's formula yields that
\begin{equation*}
\int_\Omega \epsilon\nabla \phi^\prime \cdot \nabla v \dx+ \int_{\partial\Omega}\epsilon \nabla \phi \cdot \nabla v \theta_n \ds  = \int_\Omega \bm z^{\rm T} \bm c^\prime v \dx+ \int_{\partial\Omega}\bm z^{\rm T} \bm c v\theta_n \ds.
\end{equation*}
Taking the test function $v$ with compact support in $\Omega$, then by density property and integration by parts, we obtain the shape derivative of electric potential satisfying
\begin{equation*}
-\epsilon\Delta \phi^\prime = \bm z^{\rm T} \bm c^\prime,\qquad x\in \Omega.
\end{equation*}
Taking $\frac{\partial v}{\partial \bm n} = 0$ on $\partial\Omega$, then we obtain by the tangential Green's formula that
\begin{equation}
\begin{aligned}
\int_{\partial\Omega} \epsilon \frac{\partial \phi^{\prime}}{\partial \bm n} v \ds &= \int_{\partial\Omega}\bm z^{\rm T} \bm cv\theta_n-\epsilon \nabla \phi \cdot \nabla_{\Gamma} v \theta_n  \ds \\
&=\int_{\partial\Omega}\left[\bm z^{\rm T} \bm c\theta_n +{\rm div}_{\Gamma}(\epsilon\theta_n \nabla \phi)  \right]v \ds,
\end{aligned}
\end{equation}
yielding the boundary condition on $\Gamma_{\rm in} \cup \Gamma_1$. {  The material and tangential derivatives on $\Gamma_2$ give that $\dot{\phi} = 0$ and $\nabla_\Gamma \phi = 0$. Therefore, the boundary condition for the shape derivative of electric potential is given by
\begin{equation*}
    \phi^\prime = \dot{\phi} - \nabla \phi \cdot \vec{\theta} =- \frac{\partial \phi }{\partial \bm n} {\theta}_n\quad {\rm  on}\ \Gamma_2.
\end{equation*}}
This completes the proof of the system (\ref{shapeDerState}) for the Eulerian derivative of states w.r.t. the deformation field.
\end{proof}

Next, we shall deduce the Eulerian derivative in the boundary expression, which requires the boundary to be of the $C^2$ class.  
\begin{Theorem}\label{dJBouThm}
Let $\Omega$ be of the class $C^2$ or a convex { polygon}. Suppose that $(\bm c,\phi)$ and $(\bm s,\psi)$ are the solutions of the PNP system (\ref{PNP}) and the adjoint problem (\ref{adjPro}), respectively. Then the Eulerian derivative of the shape functional { (\ref{CostFun})} reads
{  
\begin{equation}\label{dJBou}
    {\rm d}_\mathcal{S}\mathcal{J}(\Omega,\vec{\theta})=\int_{\Gamma_2} \bigg[\sum_{i=1}^N {j}(c_i)-\big({\rm D}_\Gamma \bm c + \bm Z \bm C {\rm D}_\Gamma \Phi\big):{\rm D}_\Gamma \bm s-\epsilon\frac{\partial \phi}{\partial \bm n}\frac{\partial \psi}{\partial\bm n} \bigg]{\theta}_{n} \ds,
\end{equation}}
where $\bm Z$ and $\bm C$ are defined in \eqref{ZCPHI} and ${\rm D}_\Gamma \bm w:= [ \nabla_{\Gamma} w_1, \nabla_{\Gamma} w_2,\cdots, \nabla_{\Gamma} w_N]$ for $\bm w= [w_1, w_2,\cdots, w_N]$.
\end{Theorem}
\begin{proof}
{  We formally present the derivation of Eulerian derivative in a boundary formulation below.} By Lemma \ref{shapeDerLemma}, the weak form of (\ref{ss1}) with test functions $(\bm s,\psi)\in \textbf{H}^1_{\Gamma_{\rm in},0}(\Omega)\times H^1_{\Gamma_2,0}(\Omega)$ solution of \eqref{adjPro} gives
\begin{equation}\label{SenAnaS1}
    \int_\Omega (\nabla c_i^\prime + z_i c_i^\prime \nabla \phi+ z_i c_i \nabla \phi^\prime)\cdot \nabla s_i \dx+ \int_{\partial\Omega} (\nabla c_i + z_i c_i \nabla \phi)\cdot \nabla s_i {\theta}_n \ds=0
\end{equation}
for each $i=1,2,\cdots,N$. Then, multiplying $\psi$ on both sides of the second equation in \eqref{shapeDerState} and taking the integral on $\Omega$ yield
\begin{equation}\label{dJB1}
\begin{aligned}
    \int_\Omega \bm z^{\rm T} \bm c^{\prime} \psi \dx &= \int_\Omega -\epsilon \Delta \phi^\prime \psi \dx\\
    & = \int_\Omega \epsilon \nabla \phi^\prime \cdot \nabla \psi \dx  - \int_{\Gamma_2} \epsilon \frac{\partial \phi^\prime}{\partial \bm n} \psi \ds- \int_{\Gamma_{\rm in}\cup \Gamma_1} \bm z^{\rm T} \bm c\theta_n\psi +\epsilon{\rm div}_{\Gamma}(\theta_n \nabla \phi) \psi \ds\\
    & = \int_\Omega \epsilon \nabla \phi^\prime \cdot \nabla \psi \dx,
    \end{aligned}
\end{equation}
where the boundary terms vanish owing to $\psi =0$ on $\Gamma_2$ and $\theta_n =0$ on $\Gamma_{\rm in} \cup \Gamma_1$, and the boundary condition in \eqref{shapeDerState} has been used as well. Furthermore, by Lemma \ref{adjLemma}, multiplying $(c_i^\prime, \phi^\prime)$ on both sides of \eqref{adjPro} and taking the integral on $\Omega$ yield 
\begin{equation}\label{SenAnaSE}\left\{
\begin{aligned}
& \int_\Omega -\Delta s_i c_i^\prime + z_i\nabla \phi\cdot \nabla s_i  c_i^\prime + z_i\psi  c_i^\prime  - j^\prime(c_i) c_i^\prime \dx =0,\\
&\int_\Omega -\epsilon\Delta \psi   \phi^\prime \dx =-\int_\Omega \sum_{i=1}^N \nabla\cdot (z_i c_i\nabla s_i) \phi^\prime \dx.
\end{aligned}\right.
\end{equation}
For the first equality in \eqref{SenAnaSE}, we have
\begin{equation}
\begin{aligned}
&\int_\Omega -\Delta s_i c_i^\prime + z_i\nabla \phi\cdot \nabla s_i  c_i^\prime + z_i\psi  c_i^\prime  - j^\prime(c_i) c_i^\prime \dx\\
&\qquad=\int_\Omega \nabla s_i\cdot \nabla c_i^\prime + z_i\nabla \phi\cdot \nabla s_i  c_i^\prime + z_i\psi  c_i^\prime  - j^\prime(c_i) c_i^\prime \dx- \int_{\partial\Omega} \frac{\partial s_i}{\partial \bm n} c_i^\prime \ds.
\end{aligned}
\end{equation}
More specifically, 
\begin{equation*}
    \int_{\partial\Omega} \frac{\partial s_i}{\partial \bm n} c_i^\prime \ds = \int_{\Gamma_{\rm in}\cup\Gamma_1 \cup \Gamma_2} \frac{\partial s_i}{\partial \bm n} c_i^\prime \ds,
\end{equation*}
where the integral on $\Gamma_1 \cup \Gamma_2$
vanishes due to $\frac{\partial s_i}{\partial \bm n}=0$ on $\Gamma_1 \cup \Gamma_2$, and the integral on $\Gamma_{\rm in}$ vanishes as well due to the fact that $$\int_{\Gamma_{\rm in}} \frac{\partial s_i}{\partial \bm n} c_i^\prime \ds = \int_{\Gamma_{\rm in}} \frac{\partial s_i}{\partial \bm n} (\nabla c_i-\nabla c_i^\infty) \cdot \bm \theta\ds$$ with $\bm \theta =0 $ on $\Gamma_{\rm in}$. By the Hadamard's formula and the weak formulation (\ref{SenAnaSE}), we have
\begin{equation}\label{ss3}
\begin{aligned}
{\rm d}_\mathcal{S}\mathcal{J}(\Omega,\vec{\theta}) &= \sum_{i=1}^N \int_\Omega j^\prime(c_i) c_i^\prime \dx + \sum_{i=1}^N \int_{\partial\Omega} {j}(c_i) {\theta}_{n} \ds \\
& =\sum_{i=1}^N \int_\Omega \nabla s_i\cdot \nabla c_i^\prime + z_i\nabla \phi\cdot \nabla s_i  c_i^\prime + z_i\psi  c_i^\prime \dx+ \sum_{i=1}^N  \int_{\partial\Omega} {j}(c_i) {\theta}_{n} \ds\\
& =\sum_{i=1}^N \int_\Omega \nabla s_i\cdot \nabla c_i^\prime + z_i\nabla \phi\cdot \nabla s_i  c_i^\prime\dx +\int_\Omega \epsilon \nabla \phi^\prime\cdot \nabla \psi \dx+ \sum_{i=1}^N \int_{\Gamma_2} {j}(c_i){\theta}_{n} \ds,
\end{aligned}
\end{equation}
where (\ref{dJB1}) and $\theta_n=0$ on $\Gamma_{\rm in}\cup \Gamma_1$ have been used for the last equality. It follows by integration by parts that{  
\begin{equation}\label{ss2}
\begin{aligned}
\int_\Omega \epsilon \nabla \phi^\prime\cdot \nabla \psi \dx  =& -\int_\Omega \epsilon \Delta \psi \phi^\prime \dx + \int_{\partial\Omega}\epsilon\frac{\partial \psi}{\partial \bm n}\phi^\prime \ds \\
=&\int_{\partial\Omega}\epsilon\frac{\partial\psi}{\partial\bm n}\phi^\prime\ds - \int_\Omega \sum_{i=1}^N \nabla\cdot (z_i c_i\nabla s_i) \phi^\prime \dx \\
=&-\int_{\Gamma_2}\epsilon\frac{\partial\psi}{\partial\bm n}\frac{\partial \phi}{\partial\bm n}{\theta}_{n}\ds + \int_\Omega \sum_{i=1}^N  (z_i c_i\nabla s_i)\cdot \nabla \phi^\prime \dx\\
&+\int_{\Gamma_1\cup \Gamma_{\rm in}}\epsilon\frac{\partial\psi}{\partial\bm n}\phi^\prime-\sum_{i=1}^N\bigg(z_i c_i \frac{\partial s_i}{\partial \bm n}\bigg) \phi^\prime \ds-\int_{\Gamma_2} \sum_{i=1}^N\bigg(z_i c_i \frac{\partial s_i}{\partial \bm n}\bigg) \phi^\prime \ds,
\end{aligned}
\end{equation}}
where the last two terms in the last equality vanish due to the boundary conditions 
$$\epsilon\frac{\partial\psi}{\partial\bm n}-\sum_{i=1}^N z_i c_i \frac{\partial s_i}{\partial \bm n}=0 \ {\rm on}\ \Gamma_1\cup\Gamma_{\rm in}\quad {\rm and}\quad \frac{\partial s_i}{\partial \bm n}=0\ {\rm on}\ \Gamma_2.$$ Consequently, we further obtain by combining with \eqref{ss2} and \eqref{ss3} that{  
\begin{equation*}
\begin{aligned}
{\rm d}_\mathcal{S}\mathcal{J}(\Omega,\vec{\theta}) &=\sum_{i=1}^N \int_\Omega \nabla s_i\cdot \nabla c_i^\prime + z_i\nabla \phi\cdot \nabla s_i  c_i^\prime+(z_i c_i\nabla s_i)\cdot \nabla \phi^\prime\dx+ \int_{\Gamma_2} \sum_{i=1}^N {j}(c_i){\theta}_{n} -\epsilon\frac{\partial\psi}{\partial\bm n}\frac{\partial\phi}{\partial\bm n}{\theta}_{n}\ds\\
&=-\sum_{i=1}^N\int_{\partial\Omega} (\nabla c_i + z_i c_i \nabla \phi)\cdot \nabla s_i {\theta}_n \ds+\int_{\Gamma_2} \sum_{i=1}^N {j}(c_i){\theta}_{n} -\epsilon\frac{\partial\psi}{\partial\bm n}\frac{\partial\phi}{\partial\bm n}{\theta}_{n}\ds\\
&=\int_{\Gamma_2} -\sum_{i=1}^N (\nabla_{\Gamma} c_i + z_i c_i \nabla_{\Gamma} \phi)\cdot \nabla_{\Gamma} s_i {\theta}_n +\sum_{i=1}^N {j}(c_i){\theta}_{n} -\epsilon\frac{\partial\psi}{\partial\bm n}\frac{\partial\phi}{\partial\bm n}{\theta}_{n}\ds,
\end{aligned}
\end{equation*}}
where we have utilized $\frac{\partial c_i}{\partial \bm n} + z_ic_i \frac{\partial\phi}{\partial \bm n} = 0$ on $\Gamma_2$, $\frac{\partial s_i}{\partial \bm n}=0$ on $\Gamma_2$, and $\theta_n=0$ on $\Gamma_{\rm in}\cup\Gamma_1$. The proof for the Eulerian derivative in the boundary formula (\ref{dJBou}) is completed after some matrix and vector calculations.
\end{proof}
{  
\begin{remark}
If the assumptions in \cite[Proposition 1]{LaurainJMPA} hold, one can transform the Eulerian derivative of the domain expression into a tensor formulation:
\begin{equation}
{\rm d}_\mathcal{V}\mathcal{J}(\Omega,\vec{\theta})=\int_\Omega \mathcal{S}_1 :\md \bm \theta \dx,
\end{equation}
where 
\begin{equation}
\begin{aligned}
\mathcal{S}_1= &\mathbb{I}\sum_{i=1}^N (-z_i c_i + \nabla c_i \cdot \nabla s_i+z_i c_i \nabla\phi \cdot \nabla s_i)+\mathbb{I}[\epsilon \nabla \phi\cdot \nabla\psi- (\sum_{i=1}^N z_i c_i)\psi]\\
&-\sum_{i=1}^N \nabla c_i \otimes \nabla s_i-\sum_{i=1}^N\nabla s_i \otimes \nabla c_i -\sum_{i=1}^N z_i c_i \nabla\phi\otimes \nabla s_i -\sum_{i=1}^N z_i c_i  \nabla s_i\otimes \nabla\phi\\
&- \epsilon \nabla\phi\otimes \nabla \psi - \epsilon \nabla\psi\otimes \nabla \phi.
\end{aligned}
\end{equation}
Since $\Omega$ is of $C^2$ class or a convex polygon, the Eulerian derivative in boundary form can be directly obtained by using the fact $\psi = 0$ on $\Gamma_2$:
\begin{equation}
\begin{aligned}
{\rm d}_\mathcal{S}\mathcal{J}(\Omega,\vec{\theta}) =& \int_{\partial\Omega} (\mathcal{S}_1 \bm n\cdot \bm n)\theta_n\ds\\
=&\int_{\partial\Omega} \left[\sum_{i=1}^N j(c_i)- (\sum_{i=1}^N z_i c_i)\psi  -\sum_{i=1}^N(\nabla c_i + z_i c_i \nabla \phi)\cdot \nabla s_i- \epsilon \nabla\phi\cdot \nabla\psi\right] \theta_n\ds\\
=&\int_{\Gamma_2} \left[\sum_{i=1}^N j(c_i)  -\sum_{i=1}^N(\nabla_\Gamma c_i + z_i c_i \nabla_\Gamma \phi)\cdot \nabla_\Gamma s_i- \epsilon \frac{\partial\phi}{\partial\bm n} \frac{\partial\psi}{\partial\bm n}\right] \theta_n\ds,
\end{aligned}
\end{equation}
which is identical to the result in Theorem \ref{dJBouThm}.
\end{remark}
}

\section{Optimization algorithm}
In this section, we propose a numerical shape gradient algorithm to address the shape optimization problem \eqref{cost}. The volume constrained optimization model (\ref{cost}) { is} transformed into
an unconstrained optimization model by using the augmented Lagrangian method \cite{Schulz2016}. 
We introduce an augmented Lagrangian 
\begin{equation}\label{AugLagrangian}
 \mathbb{L}(\Omega):=\mathcal{J}(\Omega,\bm c(\Omega)) + l(\mathcal{P}_1(\Omega)-\mathcal{C}_1)+\frac{1}{2}\beta(\mathcal{P}_1(\Omega)-\mathcal{C}_1)^2,
\end{equation}
where $l>0$ denotes a Lagrange multiplier and $\beta>0$ is a penalty parameter. The Eulerian derivative of $\mathcal{P}_1(\Omega)$ at $\Omega$ in the direction $\bm \theta$ is given by
\begin{equation*}
\mathcal{P}_1^\prime(\Omega;\bm \theta):=\lim_{t\rightarrow 0^+}\frac{\mathcal{P}_1(\Omega_t)-\mathcal{P}_1(\Omega)}{t}=\int_{\Omega} {\rm div}\bm\theta\dx,
\end{equation*}
{ By combining Theorem \ref{dJVolThm} and Theorem \ref{dJBouThm}, we derive the Eulerian derivatives of the augmented Lagrangian \eqref{AugLagrangian} in both distributed and boundary forms}
\begin{equation*}
\begin{aligned}
{\rm d}_\mathcal{V}\mathbb{L}(\Omega,\vec{\theta})={\rm d}_\mathcal{V}\mathcal{J}(\Omega;\bm \theta) +\int_\Omega \big[l+\beta(\mathcal{P}_1(\Omega)-\mathcal{C}_1)\big] {\rm div}\vec{\theta}\dx
\end{aligned}
\end{equation*}
and
\begin{equation*}
{\rm d}_\mathcal{S}\mathbb{L}(\Omega;\vec{\theta})={\rm d}_\mathcal{S}\mathcal{J}(\Omega;\bm \theta)+\int_{\partial\Omega}  \big[ l+\beta(\mathcal{P}_1(\Omega)-\mathcal{C}_1) \big]{\theta}_{n} \ds,
\end{equation*}
respectively. The initial guess for the Lagrange multiplier (denoted by $l_0$) is approximated from the first-order necessary optimality { condition by considering} the deformation field with unit outward normal $\bm \theta = \bm n$ on $\partial\Omega$, { which yields}
\begin{equation*}
l_0 =- \frac{{\rm d}_\mathcal{S}\mathcal{J}( \Omega; \bm n)}{\vert \Gamma_2\vert}.
\end{equation*}
{ To enforce the volume constraint throughout the optimization process, we employ a Uzawa-type update scheme for the Lagrange multiplier:}
\begin{equation}\label{Uzawa}
l \leftarrow l + \beta (\mathcal{P}_1(\Omega)-\mathcal{C}_1).
\end{equation}

Denote the space of the deformation field as $\mathcal{X}:=\{ w \in {H}^1(\Omega)| \  w=0\ {\rm on}\ \Gamma_{\rm in}\cup \Gamma_1 \}$. The descent direction can be obtained by solving the \emph{$H^1$ shape gradient descent flow} (see \cite{DFO}) with Eulerian derivative on the right-hand side: find $\vec{\zeta} \in \mathcal{X}^d$ such that
\begin{equation}\label{H1direction}
\int_{\Omega}(\epsilon_0{\rm D}\vec{\zeta}:{\rm D}\vec{\theta}+\vec{\zeta}\cdot \vec{\theta}) \dx =-{\rm d}\mathbb{L}(\Omega;{\vec{\theta}})\quad\forall\ \vec{\theta}\in \mathcal{X}^d,
\end{equation}
where ${\rm d}\mathbb{L}(\Omega;{\vec{\theta}})$ corresponds to the distributed (or boundary type of) Eulerian derivative and $\epsilon_0>0$ is a diffusion parameter. During the grid motion in the shape gradient algorithm, the mesh may become low quality even if the initial mesh is quasi-uniform. One strategy for improving mesh quality is to use uniform remeshing. The remeshing technique by the Delaunay triangulation could be costly for very fine meshes. For 2d cases, we can use a special moving mesh method by virtue of conformal transformations (CT) \cite{CR} to preserve the angles of triangles during the optimization process. More precisely, introduce $$\mathcal{B}:=\left(\begin{array}{cc}-\partial_{x} & \partial_{y} \\ \partial_{y} & \partial_{x}\end{array}\right)$$ and the symmetric part of ${\rm D}\vu$ $$\operatorname{sym}(\mathrm{D} \mathbf{u}):=\frac{1}{2}\left(\mathrm{D} \mathbf{u}+\mathrm{D} \mathbf{u}^{{\rm T}}\right).$$ { We} consider a CT-H(sym) gradient flow { defined as follows}: find $\vec{\zeta}\in \mathcal{X}^2$ such that
\begin{equation}\label{CRflow}
\int_{\Omega}\bigg(\frac{1}{\alpha}\mathcal{B}\vec{\zeta}\cdot\mathcal{B}\bm \theta+{\rm sym}({\rm D}\vec{\zeta}):{\rm sym}({\rm D}\bm \theta)+\vec{\zeta} \cdot \bm\theta \bigg)\dx  = 
- {\rm d} \mathbb{L}(\Omega;\bm \theta) \quad \forall\ \bm\theta \in  \mathcal{X}^2,
\end{equation}
where $\alpha>0$ is a parameter to control the effect of CT. With such a strategy, we are able to seek
the optimal domain of interest while maintaining the quasi-uniformity and regularity of the mesh during morphology evolution.
{\color{purple}
\begin{remark}
The gradient flow dynamical system may exhibit divergence in infinite-dimensional spaces. For a model problem addressing convergent gradient flow systems,  the readers are referred to \cite{SokoJGA2023}. The approach involves regularizing the objective functional using geometric energy functionals, such as the Willmore functional. Proper regularization of the cost functional can ensure convergence at the continuous level for gradient flow systems. Additionally, when a perimeter penalization is incorporated into shape optimization problems, positive results regarding the existence of optimal shapes have been established. In the subsequent numerical sections, we incorporate perimeter regularization to ensure that the evolving shapes remain within the class of Caccioppoli sets, thereby maintaining well-posedness and stability of the optimization process.
\end{remark}}

\section{Numerical schemes}

\subsection{Finite-element discretization of state and adjoint}
We consider a family of quasi-uniform { triangulations} $\{ \mathcal{T} _h\}_{h >0}$ of $\Omega$ into triangles { or} tetrahedrons satisfying $\overline{\Omega_h} = \bigcup _{K \in\mathcal{T} _h} \overline{K}$, where the mesh size $ h:= \max_{K\in\mathcal{T} _h} h_K$  with  $h_K$ being the diameter of any $K \in \mathcal{T}_h$. { Assume} that the discrete domain $\Omega_h$
is a polygon. { We introduce} a finite-dimensional subspace $W_h\subset {H}^1(\Omega_h)$, where the space $W_h=\{q_h\in  C^0(\overline{\Omega_h}) \, \big\lvert \,q_h|_K\in \mathbb{P}_1\ \forall K\in\mathcal{T}_h\}$ consists of continuous piece-wise linear polynomials. 
A triangulation is strongly acute \cite{Prohl2009} if the sum of opposite angles to each common side { or} face of adjacent triangles is less than $\pi-\theta_0$, with $\theta_0>0$ independent of the mesh size $h$.  

It is numerically challenging to obtain a { reliable} numerical solution of the PNP system (\ref{PNP}) when convection dominates over diffusion. We employ the so-called Gummel \cite{Gummel1964} iteration, a nonlinear block iterative algorithm that splits the PNP system (\ref{PNP}) into a Poisson--Boltzmann-type equation for the electric potential and a self-adjoint continuity system for ionic concentrations. More specifically, we introduce the Slotboom transformation \cite{FBrezzi}:
\begin{equation}\label{SoltTrans}
\rho_{i} = c_{i} \exp{(z_i \phi)}.
\end{equation}
In finite-element approximation, we consider both potential and concentration functions discretized by piece-wise linear polynomials. { We define the discrete function space} $W_{g,h}:=\{q_h\in W_h\big \vert q_h = g\ {\rm on}\ \Gamma_2 \}$ for the discrete potential $\phi_h$, $\textbf{W}_h:= (W_h)^N$, and  $\textbf{W}_{\bm \rho,h}:=\{\bm \rho_h\in \textbf{W}_h \big\vert  \rho_{i,h} =  c_i^\infty\exp{(z_i \phi_h)} \ {\rm on}\ \Gamma_{\rm in} \}$ the vector set for Slotboom variables, where $\bm \rho_h = [\rho_{1,h},\rho_{2,h},\cdots,\rho_{N,h}]^{\rm T}$. { We} introduce the discrete function spaces $\textbf{W}_{0,h}:=\{ \bm w_h\in\textbf{W}_{h}| \bm w_h=\bm 0\ {\rm on}\ \Gamma_{\rm in} \}$ and $W_{0,h}:=\{ w_h\in W_h| w_h=0\ {\rm on}\ \Gamma_{2} \}$ for the discrete adjoint variables. Then the discrete PNP system { is formulated as}: find $(\bm \rho_h,\phi_h)\in \textbf{W}_{\bm \rho,h} \times W_{g,h}$ such that
\begin{equation}\label{PNPdis}\left\{
\begin{aligned}
&\int_{\Omega_h} \exp{(-z_i\phi_h)}\nabla \rho_{i,h}\cdot \nabla v_{i,h} \dx =0,\quad&& \forall\ \bm v_h\in \textbf{W}_{0,h},\ i=1,2,\cdots,N, \\
&\int_{\Omega_h} \epsilon \nabla \phi_h \cdot \nabla \zeta_h \dx = \int_{\Omega_h} \bigg(\sum_{i=1}^N z_i \rho_{i,h} \exp{(- z_i \phi_h)}\bigg) \zeta_h \dx,\quad && \forall\ \zeta_h\in W_{0,h},
\end{aligned}\right.
\end{equation}
where the vector test function $\bm v_h:=[v_{1,h}, v_{2,h},\cdots, v_{N,h}]^{\rm T}$.  To solve the nonlinear system efficiently, we { employ} the well-known Gummel fixed-point scheme~\cite{Gummel1964} that features good convergence properties, even with a badly chosen initial guess. It is implemented as follows:

\textbf{Gummel fixed-point scheme for the steady-state Poisson--Nernst--Planck system:}

\emph{\textbf{Step 0}: Set $\ell=0$ and the stopping tolerance $\tau>0$. Initialize $\bm \rho_h^{\ell}$}. 

\emph{\textbf{Step 1}: Solve the Poisson--Boltzmann-type equation with  for $\phi^{\ell+1}_h\in W_{g,h}$:
\begin{equation}\label{semiPoisson}
\int_{\Omega_h} \epsilon \nabla \phi_h^{\ell+1} \cdot \nabla \zeta_h \dx = \int_{\Omega_h} \bigg(\sum_{i=1}^N z_i \rho_{i,h}^\ell \exp{(- z_i \phi_h^{\ell+1})}\bigg) \zeta_h \dx\quad \forall\ \zeta_h\in W_{0,h}.
\end{equation}}

\emph{\textbf{Step 2}: Solve the decoupled continuity system with the updated $\phi^{\ell+1}_h\in W_{g,h}$ for $\bm \rho_h^{\ell+1}\in \textbf{W}_{\bm \rho,h}$:
\begin{equation}\label{disCon}
    \int_{\Omega_h} \exp{(-z_i\phi_h^{\ell+1})}\nabla \rho_{i,h}\cdot \nabla v_{i,h} \dx =0\quad \forall\ \bm v_h \in \textbf{W}_{0,h},\quad i=1,2,\cdots,N.
\end{equation}}

\emph{\textbf{Step 3}: Check the stopping criterion $$\|\phi_h^{\ell+1}-\phi_h^\ell\|_1< \tau\ \text{ and }\  \|\bm \rho_h^{\ell+1}-\bm \rho_h^\ell\|_1< \tau.$$ If true, output the solution $(\bm \rho_h^{\ell+1},\phi_h^{\ell+1})$ and stop; Otherwise, update $\ell \leftarrow \ell+1$ and go back to \textbf{Step 1}}.

The convergence { analysis} for the Gummel fixed-point scheme in continuous functional space has been provided in \cite[page 332]{Markowich1986}.
Notice that the semilinear Poisson--Boltzmann-type equation (\ref{semiPoisson}) can be solved efficiently by Newton's iterations that are able to converge locally { at} a quadratic rate.
In large-convection scenarios, the singular coefficients $\exp{(- z_i\phi)}$ may cause severe numerical instability due to { their} sharp variations across elements. To { stabilize}, the discrete form of continuity equation (\ref{disCon}) is approximated by the inverse-average techniques~\cite{Brezzi1989,averageJCP2022}
\begin{equation}\label{InvAve}
\sum_{K\in\mathcal{T}_h}\int_K \exp{(-z_i\phi_h)}\nabla \rho_{i,h}\cdot \nabla v_{i,h} \dx\approx \sum_{K\in\mathcal{T}_h} E(-z_i\phi_h)_K \int_K \nabla \rho_{i,h}\cdot \nabla v_{i,h} \dx,
\end{equation}
where 
\begin{equation*}
E(-z_i\phi_h)_K = \bigg(\frac{1}{\vert K \vert}\int_K \exp{(z_i\phi_h)}\dx \bigg)^{-1}.
\end{equation*}
If the triangulation $\mathcal{T}_h$ is of weakly acute type, the matrix associated with (\ref{InvAve}) has been shown to be an M-matrix \cite{Brezzi1989}, exhibiting good properties on conservation and non-negativity~\cite{Ding2019, Ding2020}. 

The finite element discretization of the adjoint problem (\ref{adjPro}) is { formulated as finding} $\bm s_h=[s_{1,h},s_{2,h},\cdots, s_{N,h}]^{\rm T} \in \textbf{W}_{0,h}$ and $\psi_h \in W_{0,h}$ such that for all $(\bm v_h,\zeta_h)\in \textbf{W}_{0,h} \times W_{0,h}$,
\begin{equation}\label{Adjdis}\left\{
\begin{aligned}
&\int_{\Omega_h} \nabla s_{i,h}\cdot \nabla v_{i,h} + z_i (\nabla\phi_h\cdot \nabla s_{i,h})v_{i,h}+z_i\psi_h v_{i,h}- j^\prime(c_{i,h})v_{i,h} \dx =0, \\
&\int_{\Omega_h} \epsilon \nabla \psi_h \cdot \nabla \zeta_h \dx = \int_{\Omega_h} \bigg(\sum_{i=1}^N z_i c_{i,h}\nabla s_{i,h}\bigg)\cdot \nabla \zeta_h \dx.
\end{aligned}\right.
\end{equation}
Note that the adjoint system in the discretized form (\ref{Adjdis}) can be written as a linear algebraic system. Let $\mathcal{N}_h$ { denote} the set of all nodes in $\mathcal{T}_h$ and $\{q \}_{m=1}^{|\mathcal{N}_h|}$ { denote} the elements in $\mathcal{N}_h$. Define the nodal interpolation operator $\mathcal{I}_h: C(\overline{\Omega}) \rightarrow W_h$ by $\mathcal{I}_h \omega = \sum_{q\in \mathcal{N}_h} \omega(q)\nu_q$ for  $\omega\in C(\overline\Omega)$, where $\nu_q \in W_h$ denotes the nodal basis. Suppose the adjoint state variables are represented by a linear combination of the nodal basis as $s_h = \sum_{q \in \mathcal{N}_h} s_{q} \nu_{q}$ and $\psi_h = \sum_{q\in \mathcal{N}_h} \psi_{q} \nu_{q}$. Define a column vector  $(F_i)_m = \int_{\Omega_h}j^\prime(c_{i,h})\nu_{q_m}\dx$ for $i=1, 2, \cdots, N$ and $m=1, 2, \cdots, \vert\mathcal{N}_h \vert$. Introduce a stiffness matrix $A$, a mass matrix $B$, an asymmetric matrix $C$, and a matrix $D$ that are defined by $\big(A\big)_{m,n} = \int_{\Omega_h}\nabla \nu_{q_m}\cdot \nabla \nu_{q_n}\dx $, $\big(B\big)_{m,n} = \int_{\Omega_h} \nu_{q_m} \nu_{q_n}\dx$, $\big(C\big)_{m,n} = \int_{\Omega_h}\nabla \phi_h\cdot \nabla \nu_{q_m} \nu_{q_n}\dx$, and $\big(D_i \big)_{m,n} = \int_{\Omega_h}c_{i,h}\nabla \nu_{q_m}\cdot\nabla \nu_{q_n}\dx$ for $m,n=1, 2, \cdots, \vert\mathcal{N}_h \vert$, respectively. { Using} the Galerkin orthogonality, the algebraic system of (\ref{Adjdis}) becomes finding vectors $S_i=[s_{i,1},s_{i,2},\cdots,s_{i,\vert \mathcal{N}_h\vert}]^{\rm T} \in \mathbb{R}^{\vert\mathcal{N}_h \vert}$ for $i=1, 2, \cdots, N$ and $\Psi =[\phi_1,\phi_2,\cdots,\phi_{\vert \mathcal{N}_h\vert}]^{\rm T} \in \mathbb{R}^{\vert\mathcal{N}_h \vert}$ such that
\begin{equation}\label{AdjdisAlgb}
\left[\begin{tabular}{cccc|c}
$A+z_1C$     & $\mathbf{0}$ &$\cdots$ & $\mathbf{0}$ & $z_1 B$ \\
$\mathbf{0}$ &$A+z_2 C$     &$\cdots$ & $\mathbf{0}$ & $z_2 B$ \\
$\vdots$     &$\vdots$      &$\ddots$ & $\vdots$     & $\vdots$ \\
$\mathbf{0}$ &$\mathbf{0}$  &$\cdots$ & $A+z_N C$    & $z_N B$ \\
\hline
$z_1 D_1$    & $z_2 D_2$    &$\cdots$ & $z_N D_N$    &$-\epsilon A$
\end{tabular}\right]\left(\begin{array}{l}
S_1 \\
S_2 \\
\vdots \\
S_N\\
\Psi
\end{array}\right)=\left(\begin{array}{l}
F_1\\
F_2\\
\vdots\\
F_N\\
\mathbf{0}
\end{array}\right).
\end{equation}
\subsection{Discrete shape gradient flow}
Denote by $\Omega_{h,k}$ ($k=0,1,2,\cdots$) a domain represented by the mesh at the $k$-th iteration. It is updated by 
\begin{equation}\label{ShapeEvolution}
\Omega_{h,k+1}= \Omega_{h,k}+\delta_k{\vec{\zeta}}_{k,h},
\end{equation} 
where $\delta_k>0$ is a suitable step size and ${\vec{\zeta}}_{k,h}$ is a computed discrete descent direction.
We { employ} the Galerkin $\mathbb{P}_1$ finite element method to discretize $H^1$ gradient flow (\ref{H1direction}) or the CT-H(sym) gradient flow (\ref{CRflow}). Let $\mathcal{X}_h:=\{w_h \in W_h|\ w_h=0\ {\rm on}\ \Gamma_{\rm in}\cup \Gamma_1 \}$. The discrete descent direction $\vec{\zeta}_h \in \mathcal{X}_h^d$ for $H^1$ shape gradient flow is obtained by solving 
\begin{equation}\label{H1dis}
\int_{\Omega_h}(\epsilon_0{\rm D}\vec{\zeta}_h:{\rm D}\vec{\theta}_h+\vec{\zeta}_h\cdot \vec{\theta}_h){\rm d}x=-{\rm d}\mathbb{L}(\Omega_h;{\vec{\theta}_h})\quad\forall\ \vec{\theta}_h \in \mathcal{X}_h^d,
\end{equation}
where ${\rm d}\mathbb{L}(\Omega_h;{\vec{\theta}_h})$ is the finite-element discretization of ${\rm d}\mathbb{L}(\Omega;{\vec{\theta}})$. Such a linear system with a symmetric and positive definite coefficient matrix can be efficiently solved by the preconditioned conjugate gradient method. The vector-field gradient flow \eqref{H1dis} can be decomposed into $d$ scalar systems for fast numerical computation, especially for 3d cases. That is, the vector $H^1$ shape gradient flow is equivalent to more efficient scalar type of $H^1$ shape gradient flows \cite{LiNS2grid}: find 
$\zeta_{j,h}\in \mathcal{X}_h$ $(j=1,\cdots,d)$ such that
\begin{equation}\label{ScalarFlow}
	\int_{\Omega_h}  (\epsilon_0\nabla \zeta_{j,h}\cdot\nabla\eta_h +\zeta_{j,h}\eta_h){\rm d}x =- {\rm d}\mathbb{L}(\Omega_h;\bm{\xi}^j_h) \quad \forall\ \eta_h \in \mathcal{X}_h,
\end{equation}
where $\bm{\xi}^j_h=[0,\cdots,\eta_h,\cdots,0]^{\rm T}$ with $\eta_h$ { located} at the $j$-th component. The discrete CT-H(sym) gradient flow in 2d  reads: find $\vec{\zeta}_h\in \mathcal{X}_h^2$ such that
\begin{equation}\label{CRflowDis}
\int_{\Omega_h}\bigg(\frac{1}{\alpha}\mathcal{B}\vec{\zeta}_h\cdot\mathcal{B}\bm\theta_h+{\rm sym}({\rm D}\vec{\zeta}_h):{\rm sym}({\rm D}\bm\theta_h)+\vec{\zeta}_h \cdot \bm\theta_h\bigg){\rm d}x  = 
- {\rm d} \mathbb{L}(\Omega_h;\bm\theta_h) \quad \forall \ \bm\theta_h \in  \mathcal{X}_h^2.
\end{equation}
The whole shape optimization algorithm is summarized as follows:
\begin{algorithm}[htbp]\label{alg1}
\SetAlgoLined
\caption{Shape gradient algorithm for maximizing charge storage in supercapacitors}
\KwData{Input an initial domain $\Omega_0$, 
maximum number of iteration times $N_{\rm iter}$,
stopping tolerance $\tau$, 
initial guess for Lagrange multiplier $l_0$, and  volume target $\mathcal{C}_1$.}
\While {$n \leq N_{\rm iter}$}{
\emph{Step 1}: Update $({\bm \rho}_{h},\phi_{h})$ by solving the PNP system \eqref{PNPdis} with the Gummel fix-point scheme \\
\emph{Step 2}: Update the adjoint variable $(\bm s_h,\psi_h)$ by solving the discrete adjoint problem \eqref{Adjdis} \\
\emph{Step 3}: Calculate the deformation ${\vec{\zeta}}_{h}$ via \eqref{H1dis} or \eqref{CRflowDis} and update the mesh via \eqref{ShapeEvolution}\\
\emph{Step 4}: Update the Lagrange multiplier via \eqref{Uzawa}\\
$n \leftarrow n+1$\\
}
\end{algorithm}

\section{Numerical results}
Numerical simulations are performed with the software FreeFem++ \cite{freefem} on a computer with Intel Core i7-7820x @3.60 GHz and RAM 16.0 GB. Numerical tests are performed to demonstrate the performance of our shape optimization model and the proposed numerical Algorithm \ref{alg1} in enhancing charge storage in supercapacitors.

\begin{figure}[htbp]
\centering
\begin{tikzpicture}[scale=4.25]
\draw[thick, fill=gray!30!white] (0.0,0.0) -- (1.0,0.0) -- (1.0,1.0) -- (0.0,1.0) -- cycle;
\draw (0.5,1.05) node[scale=1.0] {$\Gamma_{1}$};
\draw (0.5,0.05) node[scale=1.0] {$\Gamma_{1}$};
\draw (0.5,0.5) node[scale=1.0] {$\Omega$};
\draw (0.075,0.5) node[scale=1.0] {$\Gamma_{\rm in}$};
\draw (1.075,0.5) node[scale=1.0] {$\Gamma_{2}$};
\draw [arrows = {-Latex[width=5pt, length=5pt]}] (0.0,-0.05) -- (1.0,-0.05);
\draw [arrows = {-Latex[width=5pt, length=5pt]}] (1.0,-0.05) -- (0.0,-0.05);
\draw [arrows = {-Latex[width=5pt, length=5pt]}] (-0.05,0.0) -- (-0.05,1.0);
\draw [arrows = {-Latex[width=5pt, length=5pt]}] (-0.05,1.0) -- (-0.05,0.0);
\draw (0.5,-0.1) node[scale=1.0] {$1L$};
\draw (-0.125,0.5) node[scale=1.0] {$1L$};
\end{tikzpicture}
\quad
\begin{tikzpicture}[scale=3.0]
\draw[thick] (0.7,1.4) -- (0.0,1.4) -- (0.0,0.0) -- (0.7,0.0);
\draw[dashed, fill=gray!30!white] (0.7,1.4) -- (0.7,0.0) -- (1.2,0.0) -- (1.2,0.2) -- (0.8,0.2) -- (0.8,0.4) -- (1.2,0.4) -- (1.2,0.6) -- (0.8,0.6) -- (0.8,0.8) -- (1.2,0.8) -- (1.2,1.0) -- (0.8,1.0)--(0.8,1.2) -- (1.2,1.2) -- (1.2,1.4)-- cycle;
\draw[thick]  (0.7,0.0) -- (1.2,0.0) -- (1.2,0.2) -- (0.8,0.2) -- (0.8,0.4) -- (1.2,0.4) -- (1.2,0.6) -- (0.8,0.6) -- (0.8,0.8) -- (1.2,0.8) -- (1.2,1.0) -- (0.8,1.0)--(0.8,1.2) -- (1.2,1.2) -- (1.2,1.4) -- (0.7,1.4);
\draw (0.35,0.7) node[scale=1.0] {$\Omega_{1}$};
\draw (0.9,0.9) node[scale=1.0] {$\Omega_{2}$};
\draw [arrows = {-Latex[width=5pt, length=5pt]}] (0.0,-0.05) -- (0.7,-0.05);
\draw [arrows = {-Latex[width=5pt, length=5pt]}] (0.7,-0.05) -- (0.0,-0.05);
\draw [arrows = {-Latex[width=5pt, length=5pt]}] (-0.05,0.0) -- (-0.05,1.4);
\draw [arrows = {-Latex[width=5pt, length=5pt]}] (-0.05,1.4) -- (-0.05,0.0);
\draw (0.35,-0.15) node[scale=1.0] {$0.7L$};
\draw (-0.2,0.7) node[scale=1.0] {$1.4L$};
\draw (0.1,0.7) node[scale=1.0] {$\Gamma_{\rm in}$};
\draw (0.35,1.45) node[scale=1.0] {$\Gamma_{1}$};
\draw (0.35,0.05) node[scale=1.0] {$\Gamma_{1}$};
\draw (1.0,0.7) node[scale=1.0] {$\Gamma_{2}$};
\end{tikzpicture}
\quad
\begin{tikzpicture}[scale=3.0]
\draw[thick] (0.7,1.4) -- (0.0,1.4) -- (0.0,0.0) -- (0.7,0.0);
\draw[dashed, fill=gray!30!white] (0.7,1.4) -- (0.7,0.0) -- (1.4,0.0) -- (1.4,1.4) -- cycle;
\draw[thick] (0.7,0.0) -- (1.4,0.0) -- (1.4,1.4) -- (0.7,1.4);
\draw (0.35,0.7) node[scale=1.0] {$\Omega_{1}$};
\draw (0.8,0.7) node[scale=1.0] {$\Omega_{2}$};
\draw [arrows = {-Latex[width=5pt, length=5pt]}] (0.0,-0.05) -- (0.7,-0.05);
\draw [arrows = {-Latex[width=5pt, length=5pt]}] (0.7,-0.05) -- (0.0,-0.05);
\draw [arrows = {-Latex[width=5pt, length=5pt]}] (-0.05,0.0) -- (-0.05,1.4);
\draw [arrows = {-Latex[width=5pt, length=5pt]}] (-0.05,1.4) -- (-0.05,0.0);
\draw (0.35,-0.15) node[scale=1.0] {$0.7L$};
\draw (-0.2,0.7) node[scale=1.0] {$1.4L$};
\draw (0.1,0.7) node[scale=1.0] {$\Gamma_{\rm in}$};
\draw (0.35,1.45) node[scale=1.0] {$\Gamma_{1}$};
\draw (0.35,0.05) node[scale=1.0] {$\Gamma_{1}$};
\draw (1.3,0.7) node[scale=1.0] {$\Gamma_{2}$};
\draw (1.05,0.3) node[circle,fill=white,draw][scale=2.0]{};
\draw (1.05,0.7) node[circle,fill=white,draw][scale=2.0]{};
\draw (1.05,1.1) node[circle,fill=white,draw][scale=2.0]{};
\end{tikzpicture}
\caption{Illustration of computational domains: a square domain (left), an irregular domain (middle), and a porous domain (right).}
\label{domains2d}
\end{figure}

To explore the effectiveness of shape gradient algorithm on different topologies of the domain, we perform numerical tests on regular, irregular, and porous regions in 2d (see Fig.~\ref{domains2d}). The shape optimization is implemented { via} the shape gradient flows, $H^1$ and CT-H(sym). Let the gray-shaded sub-region be the design domain, where the total net charge storage is calculated. Set $L=1$ for all cases unless specified otherwise. To prevent the optimization domain from exceeding the range in the y-direction, we replace the computed { descent} direction by ${\zeta_{2,h}}=\mathcal{S}(y)\zeta_{2,h}(x,y)$, where 
\begin{equation*}
\mathcal{S}(y)=\left\{
\begin{aligned}
& \frac{1}{1+\exp(-M(y-y_{\rm min}))},  \quad && y < \frac{y_{\rm min}+y_{\rm max}}{2},\\
& \frac{1}{1+\exp(M(y-y_{\rm max}))},\quad && y\geq \frac{y_{\rm min}+y_{\rm max}}{2}.
\end{aligned}\right.
\end{equation*}
Here $M$ is a sufficiently large number and $y_{\rm min}$ and $y_{\rm max}$ denote lower and upper limits of the computational domain along the y-direction, respectively. Set $M=100$ for all experiments. To prevent the mesh from twisting and { distorting}, especially for an irregular domain, the perimeter regularization term $\frac{1}{2}\gamma\vert \partial\Omega\vert^2$ is added to the objective with a small parameter $0<\gamma\ll 1$.

\subsection{Experiment 1 (Square domain)} { We} consider a design domain of a regular square (see Fig. \ref{domains2d} left). Notice that we choose the dimensionless parameter $\epsilon=1.0$ in this experiment. The initial mesh (see Fig. \ref{Case1IniOpt} upper left) has $5938$ triangular elements and $3070$ nodes.

\textbf{\emph{Case 1.1.}} We first consider only one ionic species and test the Algorithm \ref{alg1} with a boundary type of shape gradient. The Dirichlet boundary data { is specified as} $c^\infty_1=1$ on $\Gamma_{\rm in}$ with valence $z=1$ and $g=-0.75$ on $\Gamma_2$. Initialize the  domain as a square with the Lebesgue measure $\vert\Omega\vert = 1$. Take the target volume $\mathcal{C}_1=1.75$, perimeter penalty coefficient $\gamma=10^{-4}$, and $\epsilon_0=0.05$. We assume that each triangle of the initial mesh is equilateral so that we compute numerically the mesh characteristic length $\sqrt{{4{\rm meas}(\Omega)}/ ({\sqrt{3}N_e})}$, which is used for regular and quasi-uniform remeshing in FremFem++ \cite{DFO}. In the upper right plot of Fig. \ref{Case1IniOpt}, the optimal mesh computed by boundary type of $H^1$ gradient flow presents two protruding corners. See Fig. \ref{Case1IniOpt} for counterion concentrations (middle left)  with contours and electric potential (middle right) corresponding to the optimized design. As expected, { we} observed that the concentration of counterions increases significantly when approaching the electrode interface $\Gamma_2$, where the electric double layers that store most of charges are formed. Meanwhile, the electric potential gets screened by the accumulated counterions in electric double layers and becomes homogeneous in the left bulk region.
\begin{figure}[htbp]
    \begin{minipage}[b]{0.5\textwidth}
		\centering
		\includegraphics[width=1.4 in]{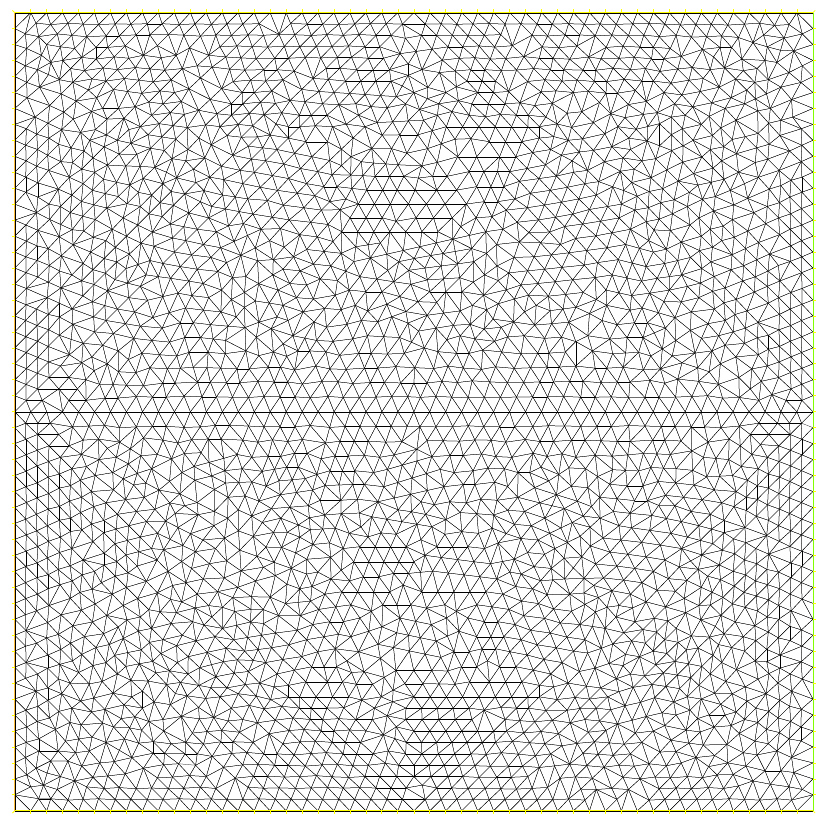}
	\end{minipage}
     \begin{minipage}[b]{0.5\textwidth}
		\centering
		\includegraphics[width=2.75 in]{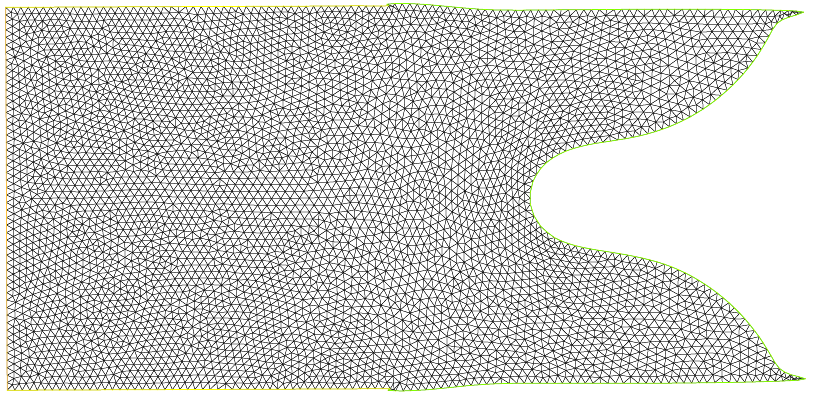}
	\end{minipage}
    \\
    \begin{minipage}[b]{0.5\textwidth}
		\centering
		\includegraphics[width=2.75 in]{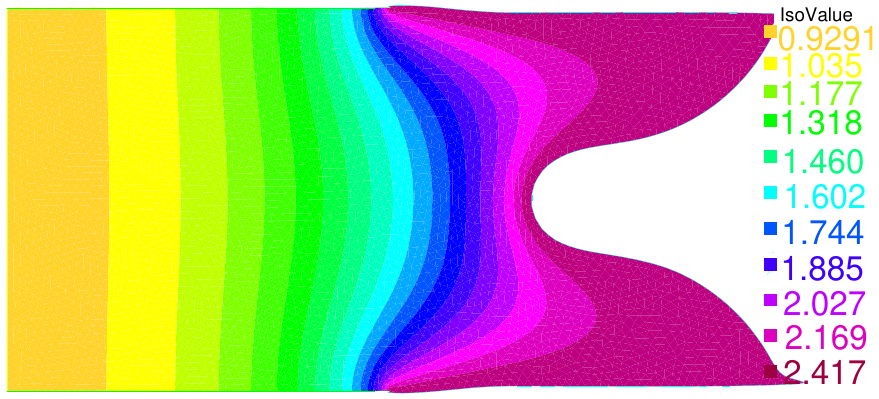}
	\end{minipage}
     \begin{minipage}[b]{0.5\textwidth}
		\centering
		\includegraphics[width=2.75 in]{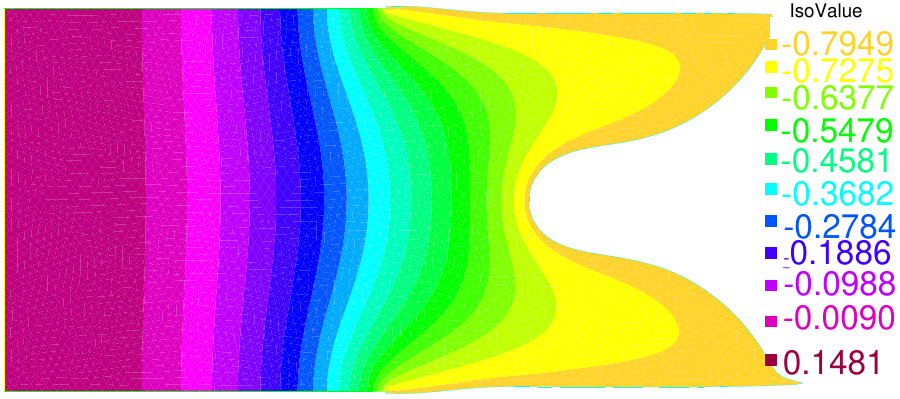}
	\end{minipage}
	\caption{Initial domain (upper left) and optimized domain with mesh (upper right), counterion concentration $c_1$ (lower left) and electric potential $\phi$ (lower right) by $H^1$ shape gradient flow for Case 1.1.}
	\label{Case1IniOpt} 
\end{figure}

\textbf{\emph{Case 1.2.}} In this case, we employ the CT-H(sym) gradient flow (\ref{CRflow}) with $\alpha=2$ using the same boundary conditions and optimization parameters as above. With preserved angles in each triangle, the optimized morphology { shown} in Fig. \ref{Case1CRBou} (upper left) is much smoother.  Middle plots of Fig. \ref{Case1CRBou} present contours of counterion concentration and electric potential computed for the optimized design. 

\textbf{\emph{Case 1.3.}} Now we test CT-H(sym) gradient flow with shape gradient in domain expression. Take the same initial morphology, and set the target volume  $\mathcal{C}_1=1.75$. The same boundary conditions and optimization parameters are used as in \emph{Case 1.1}. From the upper right plot of Fig. \ref{Case1CRBou}, the optimized domain is shown to be similar to CT-H(sym) gradient flow with boundary type of shape gradient. The results corresponding to the optimal morphology show { that} objective $\mathcal{J}=2.7559$ and volume error less than $0.01$.

From the results in \emph{Case 1.1} and \emph{Case 1.2}, it can be found that CT-H(sym) gradient flow \eqref{CRflow} gives better optimized shape with a smoother boundary. { By} comparing the results in \emph{Case 1.2} and \emph{Case 1.3}, we find that the optimized domains computed by the distributed and boundary type of Eulerian derivatives share similar morphology. From lower plots of Fig. \ref{Case1CRBou}, { we observe} that the convergence histories of the objective tend to be flat and the volume errors converge approximately to $10^{-2}$ for both gradient flows with the distributed and boundary type of shape gradients. Consequently, the following numerical experiments are performed using a CT-H(sym) gradient flow with a boundary type of shape gradient.

\begin{figure}[htbp]
    \begin{minipage}{0.5\textwidth}
		\centering
		\includegraphics[width=3.0 in]{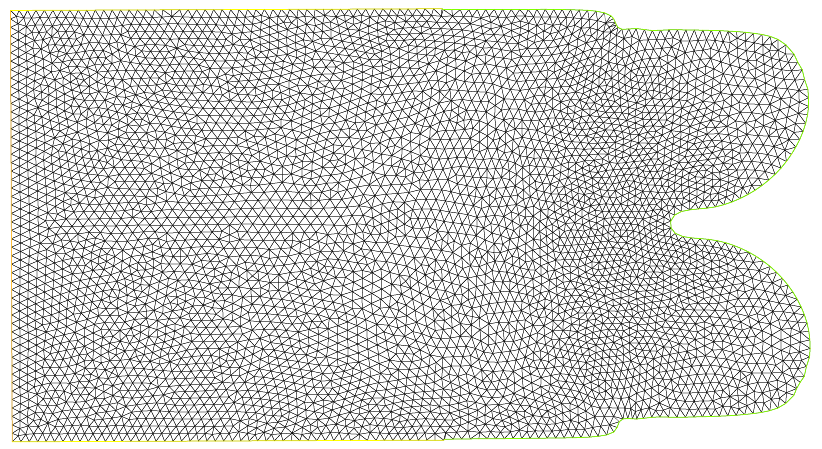}
	\end{minipage}
        \begin{minipage}{0.5\textwidth}
		\centering
		\includegraphics[width=3.0 in]{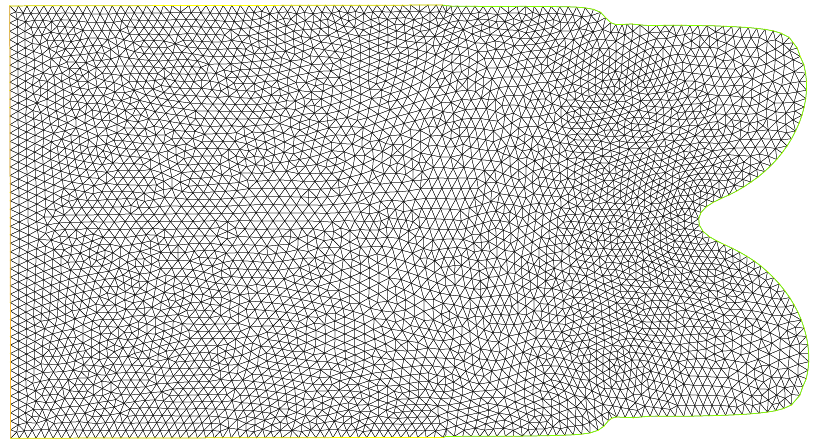}
	\end{minipage}
\\
     \begin{minipage}[b]{0.5\textwidth}
		\centering
	\includegraphics[width=3.0 in]{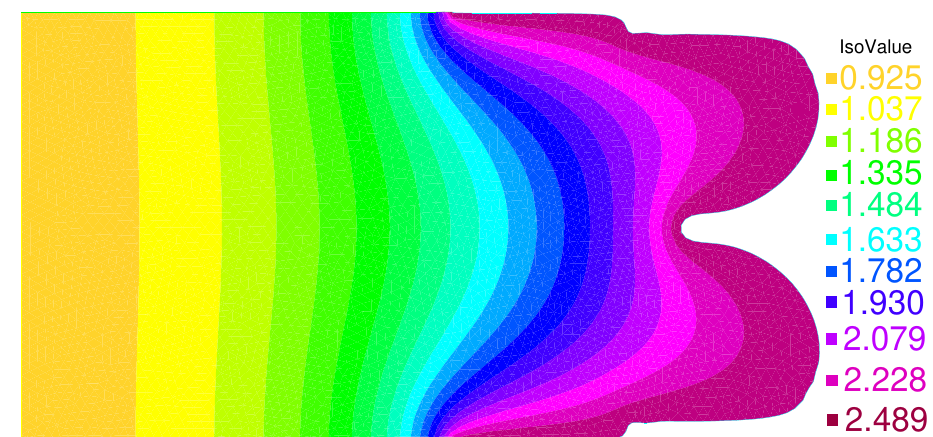}
	\end{minipage}
     \begin{minipage}[b]{0.5\textwidth}
		\centering
		\includegraphics[width=3.0 in]{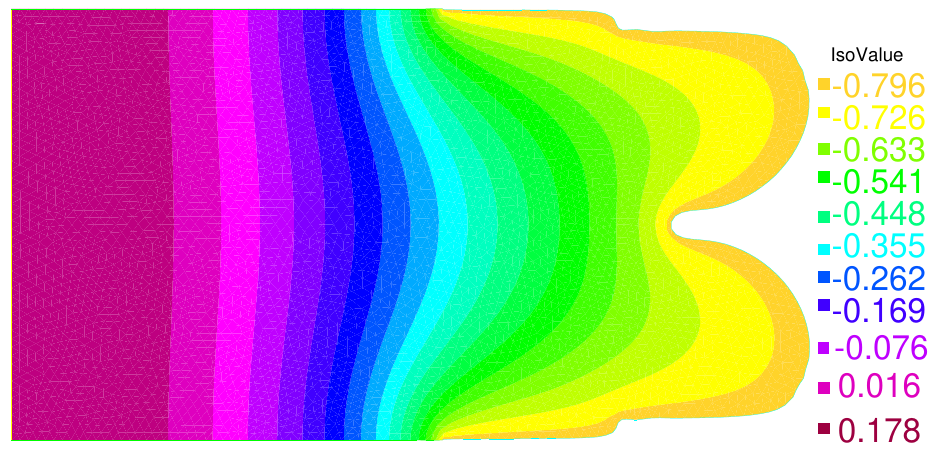}
	\end{minipage}
 \\
      \begin{minipage}[b]{0.5\textwidth}
		\centering
		\includegraphics[width=3. in]{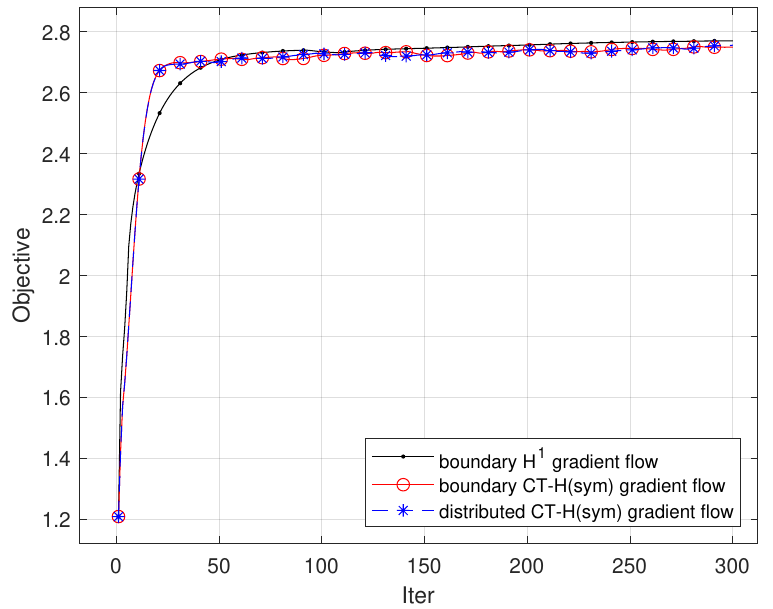}
	\end{minipage}
      \begin{minipage}[b]{0.5\textwidth}
		\centering
		\includegraphics[width=3. in]{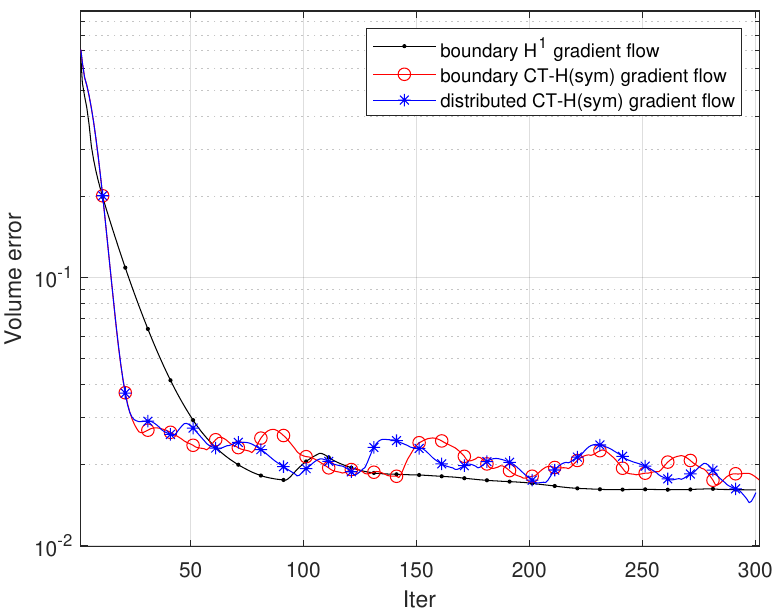}
	\end{minipage}
	\caption{Optimal mesh computed by CT-H(sym) with boundary type of shape gradient (upper left) and distributed shape gradient (upper right), counterion concentration $c_1$ (middle left) and electric potential $\phi$ (middle right) computed with the boundary type of shape gradient and CT-H(sym) gradient flow for Case 1.2, and convergence histories of objective (lower left) and volume error (lower right) for all Cases 1.1 - 1.3.}
	\label{Case1CRBou} 
\end{figure}

\subsection{Experiment 2 (Irregular domain)} In this experiment, we choose the initial domain as an irregular shape (see Fig. \ref{domains2d} middle). The objective is defined on a subdomain by $\mathcal{J}(\Omega_2,\bm c(\Omega))=\int_{\Omega_2} \sum_{i=1}^N j(c_i)\dx$. The shape sensitivity analysis is the same except that the right-hand side term in the adjoint system (\ref{adjPro}) is replaced by $-\chi_{\Omega_2} j^\prime(c_i)$, where $\chi_{\Omega_2}$ is the characteristic function of $\Omega_2$. The initial mesh (see Fig. \ref{Case2CRBouBigV0} upper left) has $6584$ triangular elements and $3518$ nodes.

In this case, we consider binary electrolytes with ionic valences $z_1=1$ and $z_2=-1$. Set the boundary condition for ionic concentrations on $\Gamma_{\rm in}$: $c_1^\infty=0.5$ and $c_2^\infty = 0.5$, and the boundary condition for electric potential on $\Gamma_2$ with $g=-1.0$. Take the simulation parameters $\epsilon=1.0, \gamma=0.2$, $\alpha = 2.5$, $\beta=5$ and $\mathcal{C}_1 = 1.6$. Numerical simulations start { from} an initial design as shown in Fig. \ref{Case2CRBouBigV0} (upper left). The convergence history of the objective presented in Fig. \ref{Case2CRBouBigV0} (upper right) { shows} that the total charge storage for the optimized morphology is almost three times compared with the initial one. Fig. \ref{Case2CRBouBigV0} (lower left) shows that cations, as counterions to the negatively { charged} electrode, are attracted into the finger-shaped optimized domain $\Omega_2$, while anions, as coions to the electrode, are repelled from the region due to electrostatic interactions. As depicted in Fig. \ref{Case2CRBouBigV0} (lower right), the electric potential is much screened by the counterions and level off in the bulk region.

\begin{figure}[htbp]
	\begin{minipage}{0.5\textwidth}
		\centering
		\includegraphics[width=2.8 in, height=2.2 in]{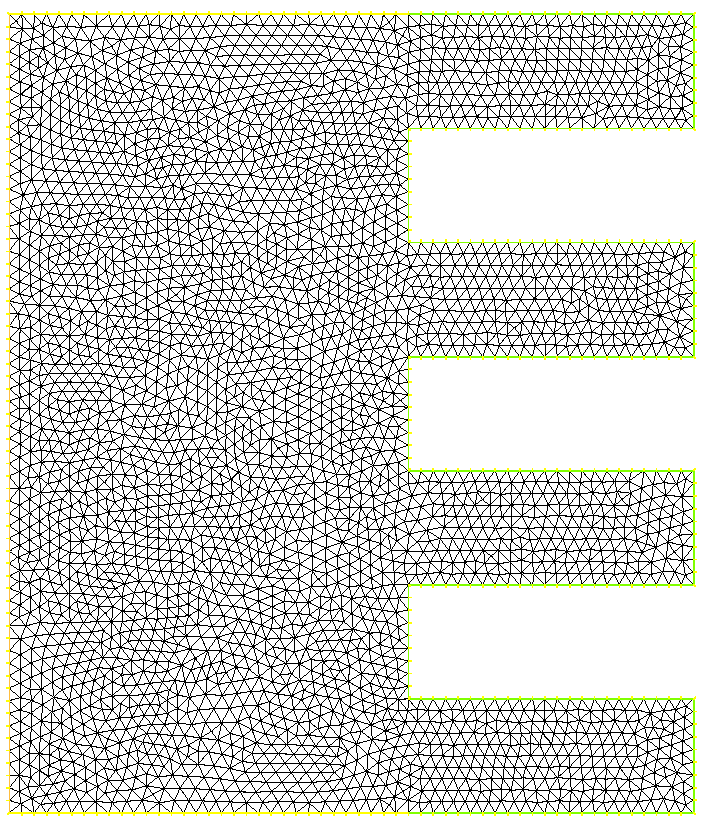}
	\end{minipage}
     \begin{minipage}{0.5\textwidth}
		\centering
		\includegraphics[width=2.9 in]{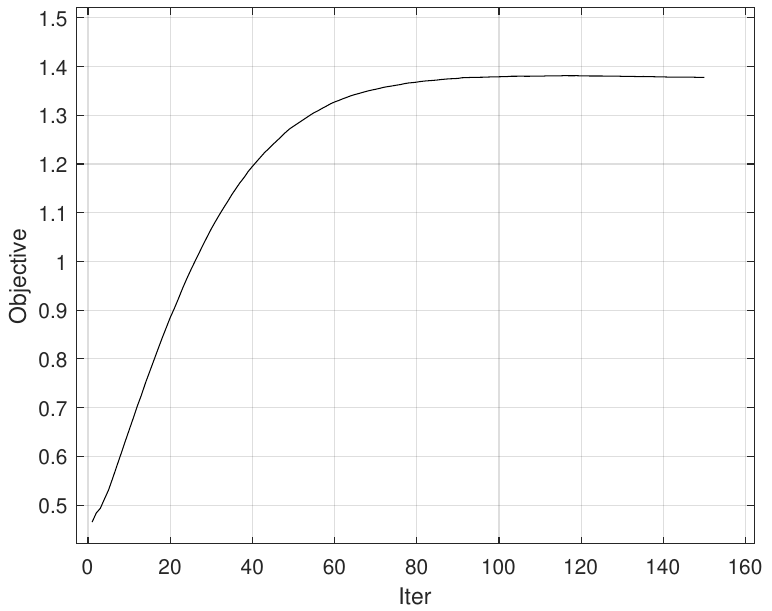}
	\end{minipage}
 \\
     \begin{minipage}{0.5\textwidth}
		\centering
		\includegraphics[width=2.8 in]{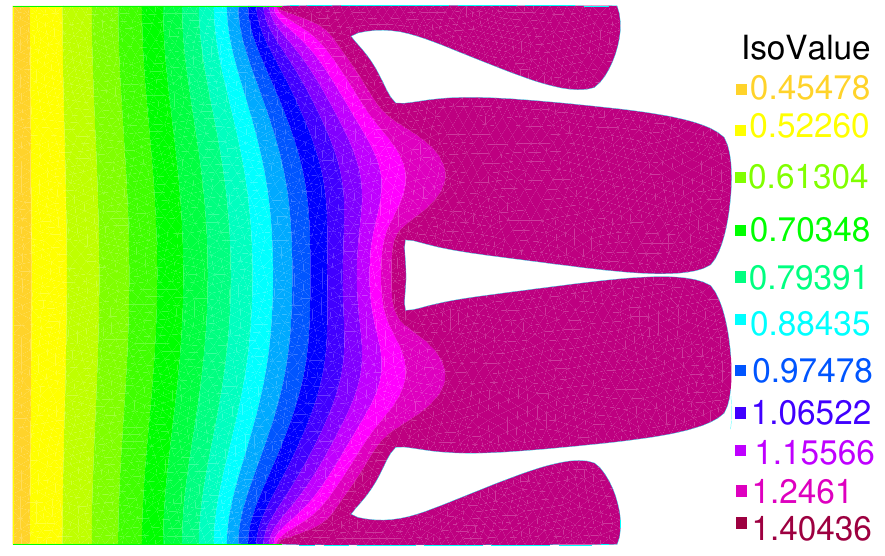}
	\end{minipage}
     \begin{minipage}{0.5\textwidth}
		\centering
		\includegraphics[width=2.8 in]{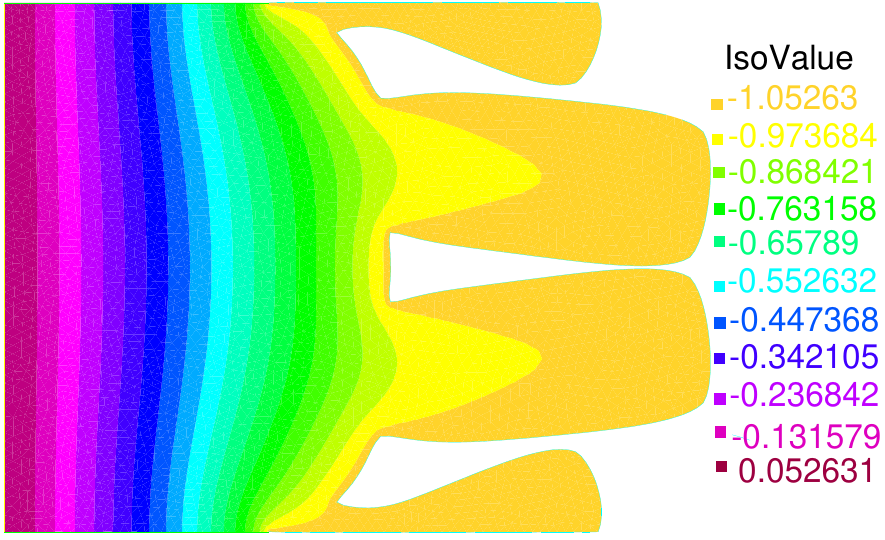}
	\end{minipage}
	\caption{Initial mesh(upper left),  convergence histories of the objective (upper right), counterion concentration $c_1$ (lower left), and electric potential $\phi$ (lower right) at the optimized shape computed by CT-H(sym) gradient flow for Experiment 2.  
    }
	\label{Case2CRBouBigV0} 
\end{figure}

\subsection{Experiment 3 (Porous domain)}
Consider a porous initial domain with six irregular shapes inside; cf. right plot of Fig. \ref{domains2d}. Binary electrolytes with ionic valences $z_1=1$ and $z_2=-1$ are taken into account. The initial mesh (see Fig. \ref{Case32} upper left) has $13842$ triangular elements and $7294$ nodes.
The target volume is set as 1.25 times the initial volume. Take the following simulation parameters: $c_1^{\infty}=c_1^{\infty}=0.5, g=-1, \epsilon=10^{-2}, \gamma=10^{-3}$, $\alpha = 2.5$ and $\beta=10$. Notice that the boundary of the holes is a part of $\Gamma_2$ that is to be optimized. The convergence history of the objective is presented in the upper right plot of Fig. \ref{Case32}, showing the effectiveness of the optimization algorithm \ref{alg1} in maximizing charge storage. The lower left plot of Fig. \ref{Case32} presents the optimized morphology of porous configuration and the corresponding distribution of counterion concentration $c_1$. The computed volume error is less than 0.01.  Clearly, it can be seen that the charged surface has large variations in the upper and lower boundaries. The counterions are accumulated in thin electric double layers, { since} the simulations take a relatively small $\epsilon$ that gives rise to sharp boundary layers. Such results are consistent with known conclusions that charge storage can be enhanced by increasing charged surface area in electrodes.  The electric potential corresponding to the optimized morphology is presented in the lower right plot of Fig. \ref{Case32}, which demonstrates that the potential gets screened drastically in sharp electric double layers. 

\begin{figure}[htbp]
	\begin{minipage}{0.5\textwidth}
		\centering
		\includegraphics[width=2.9 in, height=2.2 in]{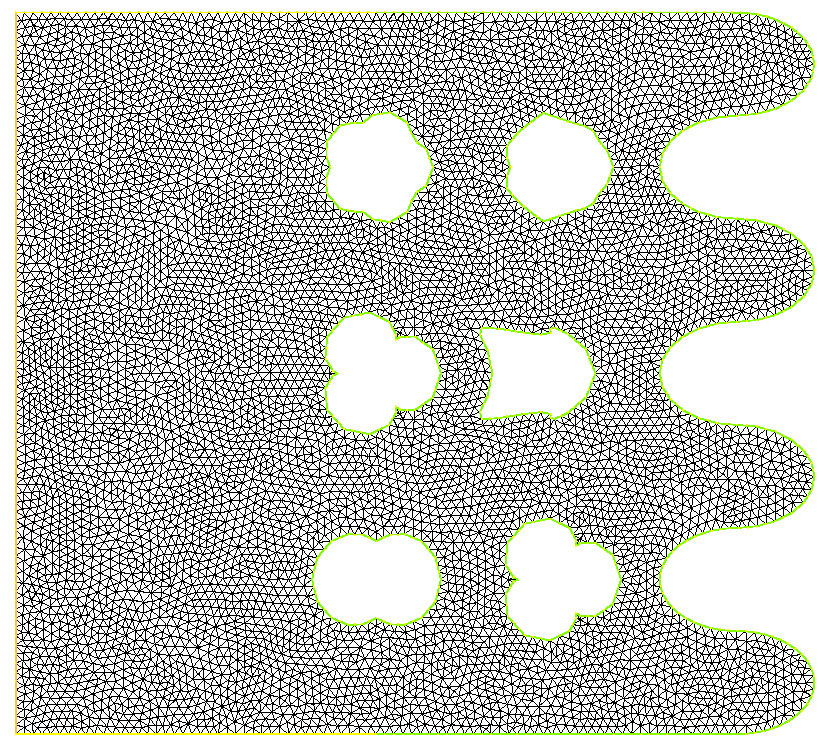}
	\end{minipage}
    \begin{minipage}{0.5\textwidth}
		\centering
		\includegraphics[width=2.9 in]{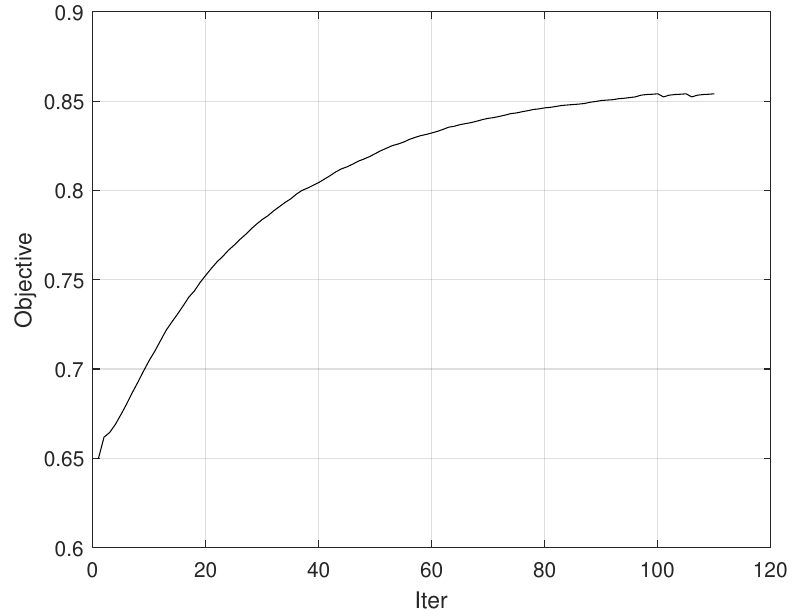}
	\end{minipage}
    \\
    \begin{minipage}{0.5\textwidth}
		\centering		\includegraphics[width=2.8 in]{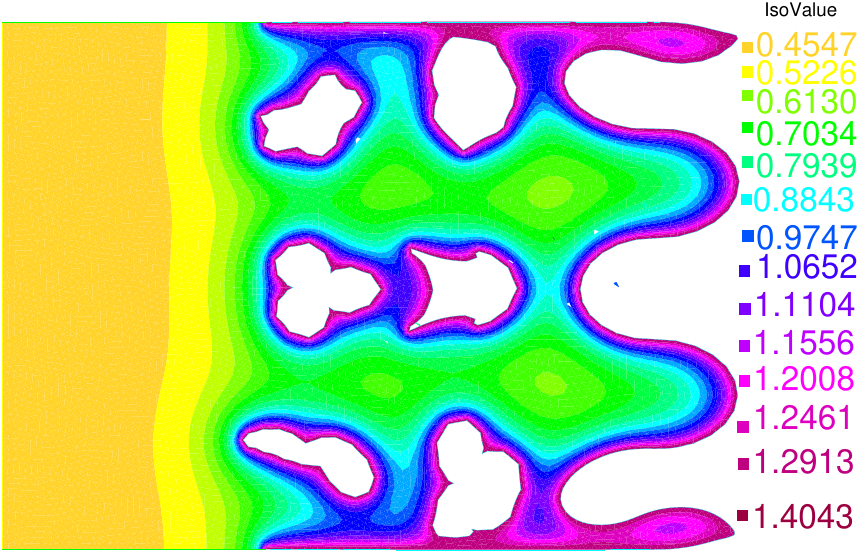}
	\end{minipage}
    \begin{minipage}{0.5\textwidth}
		\centering
		\includegraphics[width=2.8 in]{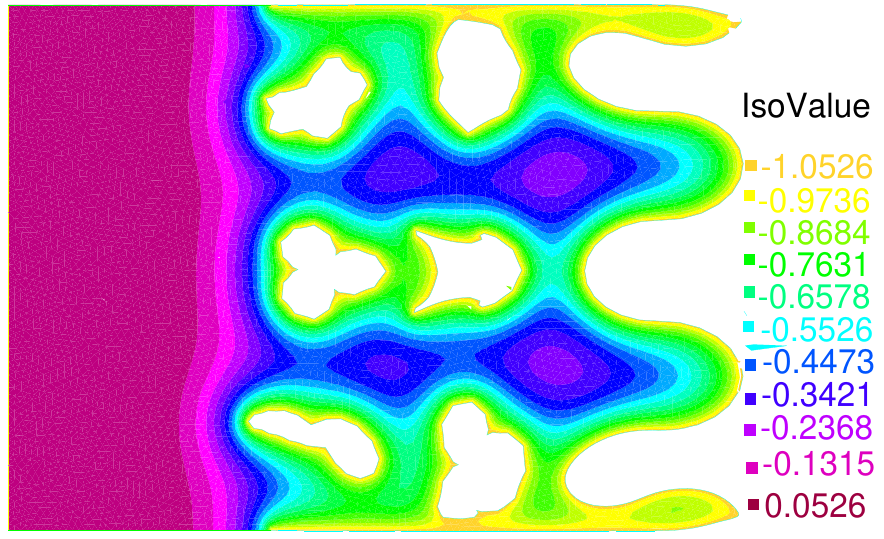}
	\end{minipage}
    \caption{Initial mesh (upper left), convergence history of the objective (upper right),  counterion concentration $c_1$ at the optimal morphology computed by CT-H(sym) gradient flow (upper right),  and potential $\phi$ of optimized morphology computed by CT-H(sym) gradient flow (lower right) for Experiment 3.}
	\label{Case32} 
\end{figure}

\begin{figure}[htbp]
\centering
\begin{tikzpicture}[scale=4.]
\draw[thick] (0.0,1.4) -- (0.0,0.0) -- (1.2,0.0) -- (1.2,0.2) -- (0.8,0.2) -- (0.8,0.4) -- (1.2,0.4) -- (1.2,0.6) -- (0.8,0.6) -- (0.8,0.8) -- (1.2,0.8) -- (1.2,1.0) -- (0.8,1.0)--(0.8,1.2) -- (1.2,1.2) -- (1.2,1.4)-- cycle;
\draw[dashed] (0.0+0.1,1.4+0.1) -- (0.0+0.1,0.0+0.1) -- (1.2+0.1,0.0+0.1) -- (1.2+0.1,0.2+0.1) -- (0.8+0.1,0.2+0.1) -- (0.8+0.1,0.4+0.1) -- (1.2+0.1,0.4+0.1) -- (1.2+0.1,0.6+0.1) -- (0.8+0.1,0.6+0.1) -- (0.8+0.1,0.8+0.1) -- (1.2+0.1,0.8+0.1) -- (1.2+0.1,1.0+0.1) -- (0.8+0.1,1.0+0.1)--(0.8+0.1,1.2+0.1) -- (1.2+0.1,1.2+0.1) -- (1.2+0.1,1.4+0.1)-- cycle;
\draw[thick] (0.0,1.4) -- (0.1,1.5);
\draw[dashed] (0.0,0.0) -- (0.1,0.1);
\draw[thick] (1.2,0.0) -- (1.3,0.1);
\draw[thick] (1.2,0.2) -- (1.3,0.3);
\draw[thick] (0.8,0.2) -- (0.9,0.3);
\draw[dashed] (0.8,0.4) -- (0.9,0.5);
\draw[thick] (1.2,0.4) -- (1.3,0.5);
\draw[thick] (1.2,0.6) -- (1.3,0.7);
\draw[thick] (0.8,0.6) -- (0.9,0.7);
\draw[dashed] (0.8,0.8) -- (0.9,0.9);
\draw[thick] (1.2,0.8) -- (1.3,0.9);
\draw[thick] (1.2,1.0) -- (1.3,1.1);
\draw[thick] (0.8,1.0) -- (0.9,1.1);
\draw[dashed] (0.8,1.2) -- (0.9,1.3);
\draw[thick] (1.2,1.2) -- (1.3,1.3);
\draw[thick] (1.2,1.4) -- (1.3,1.5);
\draw[thick] (1.3,1.5) -- (0.1,1.5);
\draw[thick] (1.3,0.1) -- (1.3,0.3);
\draw[thick] (1.3,0.5) -- (1.3,0.7);
\draw[thick] (1.3,0.9) -- (1.3,1.1);
\draw[thick] (1.3,1.3) -- (1.3,1.5);
\draw[thick] (0.9,0.3) -- (1.3,0.3);
\draw[thick] (0.9,0.7) -- (1.3,0.7);
\draw[thick] (0.9,1.1) -- (1.3,1.1);
\draw[thick] (0.9,0.3) -- (0.9,0.4);
\draw[thick] (0.9,0.7) -- (0.9,0.8);
\draw[thick] (0.9,1.1) -- (0.9,1.2);
\draw [arrows = {-Latex[width=5pt, length=5pt]}] (-0.05,0.0) -- (-0.05,1.4);
\draw [arrows = {-Latex[width=5pt, length=5pt]}] (-0.05,1.4) -- (-0.05,0.0);
\draw [arrows = {-Latex[width=5pt, length=5pt]}] (0.0,-0.05) -- (1.2,-0.05);
\draw [arrows = {-Latex[width=5pt, length=5pt]}] (1.2,-0.05) -- (0.0,-0.05);
\draw (-0.15,0.7) node[scale=1.0] {$1.4L$};
\draw (0.6,-0.15) node[scale=1.0] {$1.2L$};
\draw (0.08,0.7) node[scale=1.0] {$\Gamma_{\rm in}$};
\draw (1.36,0.75) node[scale=1.0] {$\Gamma_{2}$};
\draw (0.6,0.05) node[scale=1.0] {$\Gamma_{1}$};
\draw (0.6,1.45) node[scale=1.0] {$\Gamma_{1}$};
\draw (0.6,0.7) node[scale=1.0] {$\Omega$};
\end{tikzpicture}
\caption{Illustration for computational region in 3d.}
\label{Case4Geo} 
\end{figure}


\subsection{Experiment 4 (Irregular domain in 3d)}
In this experiment, we consider shape optimization for charge storage in 3d supercapacitors.
Note that CT-H(sym) gradient flow (\ref{CRflowDis}) is not available in { 3D}. For the sake of efficiency, we utilize the scalar $H^1$ gradient flow (\ref{ScalarFlow}) to obtain the descent direction. 

%

Consider binary electrolytes with monovalent ions. The initial domain has three ``cylinder holes" with four finger-shaped geometry; cf. the upper left plot of Fig. \ref{Exp4Case2}. The initial mesh has $69446$ nodes and $34176$ { tetrahedral} elements. Set the target volume as 1.4 times the initial volume. Simulations take $\epsilon=10^{-2}$, $\gamma=2.5\times 10^{-3}$, and $\beta=5$. The boundary data are given by $g=-1$, $c_1^\infty=0.5$, and $c_2^\infty=0.5$.  Fig. \ref{Exp4Case2} presents the optimal design of the electrode geometry starting from the initial domain, as well as the corresponding electric potential and distribution of cation concentration $c_1$ and anion concentration $c_2$. It can be seen that the figure-shaped domain enlarges as that in the 2d case, and the cylinder holes become narrower and smoother. The counterions, as expected, mainly distribute around the charge electrode surface, and the coions are repelled away from the electrode due to electrostatics.  The convergence histories for both the objective and volume error are plotted in Fig. \ref{Case4ObjVol}, which shows that the proposed optimization model with the algorithm can effectively increase charge storage of supercapacitors through shape optimization of the electrode.

\begin{figure}[htbp]

 	\begin{minipage}[b]{0.5\textwidth}
		\centering
		\includegraphics[width=2.8 in]{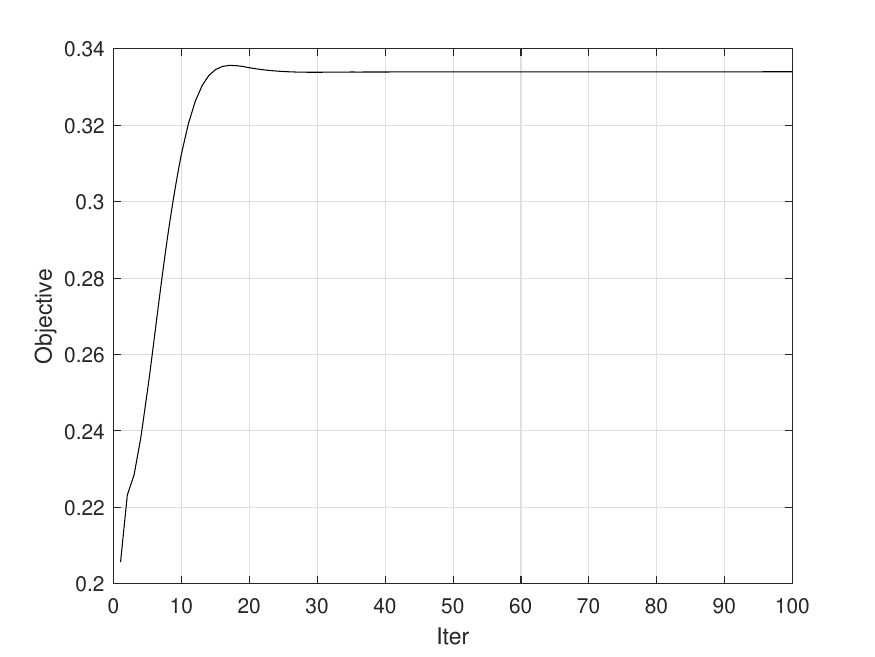}
	\end{minipage}
	\begin{minipage}[b]{0.25\textwidth}
		\centering
		\includegraphics[width=2.5 in]{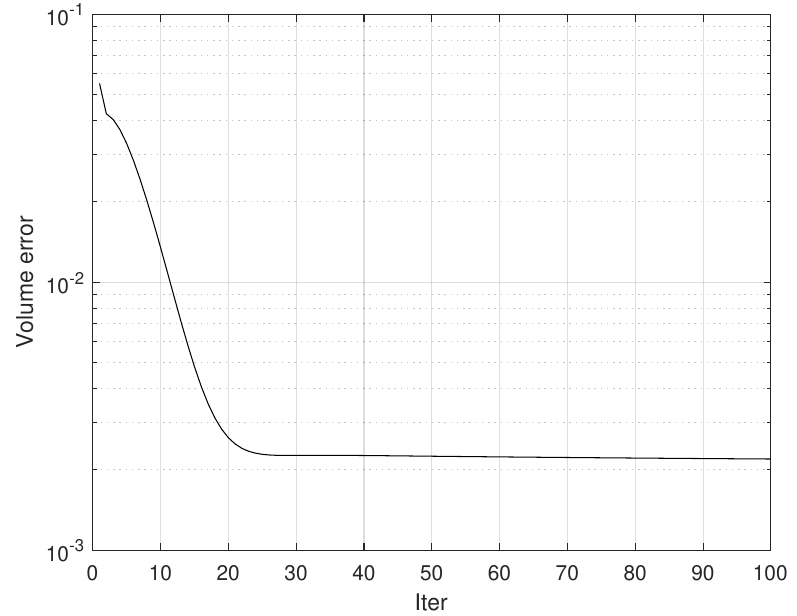}
	\end{minipage}
	\caption{The convergence histories of objective (left) and volume error (right) for Experiment 4.}
	\label{Case4ObjVol} 
\end{figure}
\begin{figure}[htbp]
	\begin{minipage}[b]{0.5\textwidth}
		\centering
		\includegraphics[width=1.5 in]{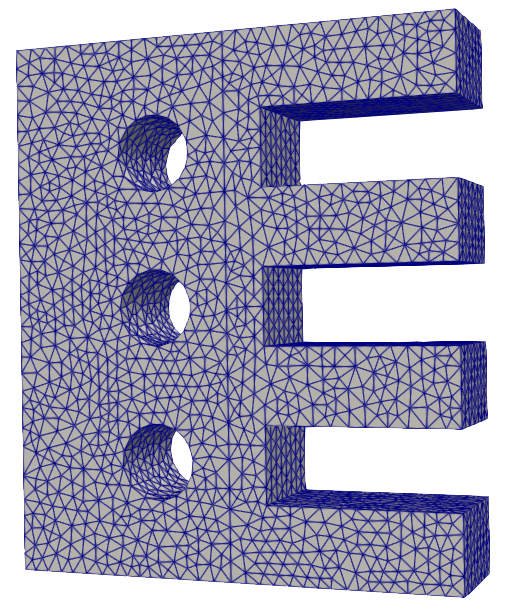}
	\end{minipage}
	\begin{minipage}[b]{0.5\textwidth}
		\centering
		\includegraphics[width=1.7 in]{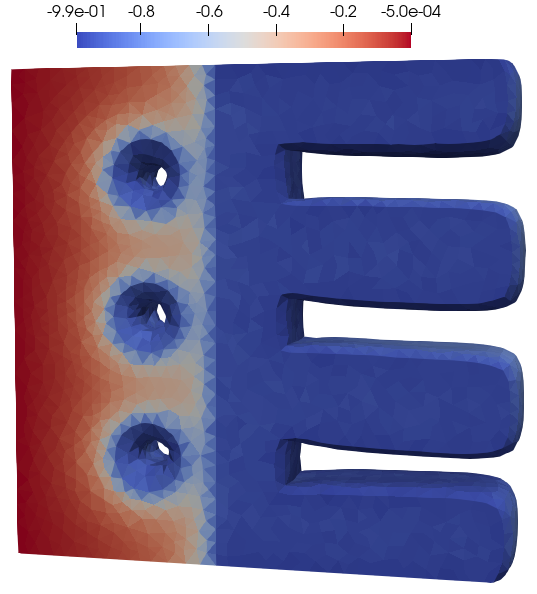}
	\end{minipage}
 \\
 	\begin{minipage}[b]{0.5\textwidth}
		\centering
		\includegraphics[width=1.7 in]{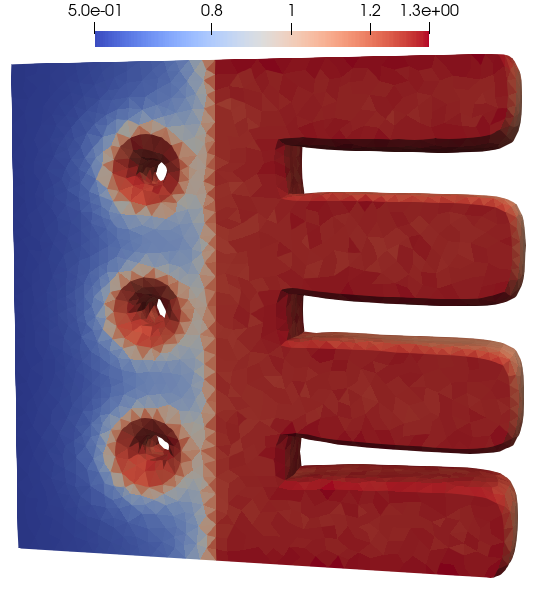}
	\end{minipage}
	\begin{minipage}[b]{0.5\textwidth}
		\centering
		\includegraphics[width=1.7 in]{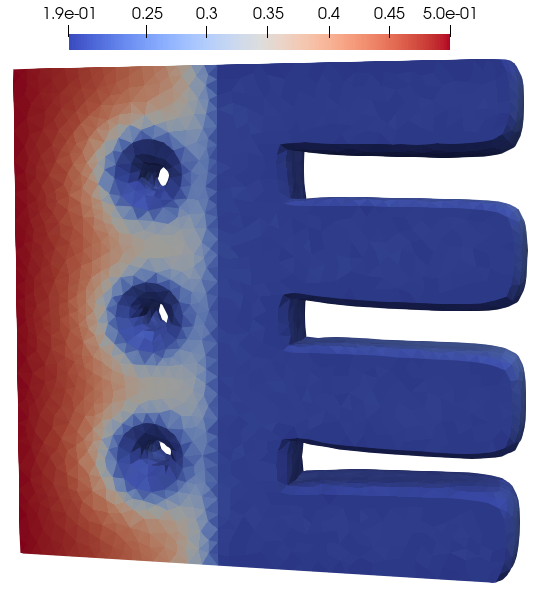}
	\end{minipage}
	\caption{Initial domain with mesh (upper left), electric potential $\phi$ (upper right), cation concentration $c_1$ (lower left), and anion concentration $c_2$ (lower right) correspond to the optimized morphology computed by $H^1$ gradient flow for Experiment 4.}
	\label{Exp4Case2} 
\end{figure}

\clearpage
\section{Conclusions}
In this work, a shape optimization model has been proposed to maximize charge storage in supercapacitors. Shape sensitivity analysis has been performed to derive Eulerian derivatives in both domain and boundary expressions. The Gummel fixed-point scheme with finite-element discretization has been developed to solve the governing PNP system that could have large convection in sharp electric double layers. Gradient flow algorithms have been used as well to find an optimal electrode morphology. Extensive numerical experiments have been conducted to demonstrate the effectiveness of the proposed optimization model and the companion numerical methods. With further refinement in electrochemical description, the optimization model and numerical methods proposed in this work provide a promising tool in the design of electrode morphology to enhance charge storage. { Future research in shape optimization of supercapacitors may include the enhancement of ion diffusion/charge transfer through geometry optimization of electrolyte-electrode interface.
}

From numerical results, one can find that topological changes are not allowed in current development of shape optimization approach. { The limitation of the current method is that shape changes need to be represented through grid movement. Hence, the number of holes and topological property can not change during shape evolution.} Empirically, electrode-electrolyte interface area that is closely related to the total charge storage can be further increased using topology optimization approach. This shall be systematically studied in our future work.


\end{document}